\documentclass[preprint,10.5pt]{elsarticle}
\usepackage{graphicx} 

\usepackage[utf8]{inputenc}
\usepackage[english]{babel}
\usepackage{graphicx}
\usepackage{amssymb}
\usepackage{amsmath}
\usepackage{enumerate}
\usepackage{tikz}
\usetikzlibrary{decorations.pathreplacing} 
\usepackage{commath}
\usepackage{upgreek}
\usepackage[hidelinks]{hyperref}
\usepackage{color}
\usepackage{wrapfig}
\usepackage{graphicx,wrapfig,lipsum}

\usepackage{amsthm}
\newtheorem{theorem}{Theorem}[section]

\newtheorem{lemma}[theorem]{Lemma}
\newtheorem{remark}{Remark}[section]

\DeclareMathOperator{\tr}{tr}

\date{\vspace{-5ex}}

\usepackage[a4paper, total={6.5 in, 9.5 in}]{geometry}

\journal{Physica D: Nonlinear Phenomena}

\begin{document}

\begin{frontmatter}

\title{Transient instability and recurrent aggregation in a two--species chemotaxis--competition system under environmental periodicity}





\author[ucm-imi,icai]{Federico Herrero-Hervás \corref{cor1}}
\ead{fedher01@ucm.es}
\author[ucm]{Pedro Hernández-Alonso}
\author[ucm-imi]{Mihaela Negreanu}
\cortext[cor1]{Corresponding author}

\affiliation[ucm-imi]{organization={Instituto de Matemática Interdisciplinar, Departamento de Análisis Matemático y Matemática Aplicada, Universidad Complutense de Madrid},
            city={Madrid},
            postcode={28040},
            country={Spain}}

\affiliation[ucm]{organization={Departamento de Análisis Matemático y Matemática Aplicada, Universidad Complutense de Madrid},
    city={Madrid},
    postcode={28040},
    country={Spain}}
\affiliation[icai]{organization={Departamento de Matemática Aplicada, ICAI, Universidad Pontificia de Comillas},
            city={Madrid},
            postcode={28015},
            country={Spain}}

\begin{abstract}
This work investigates recurrent pattern formation in a two--species chemotaxis system with competitive dynamics under the assumption that the carrying capacities of both species are periodic in time. By considering a globally attracting periodic solution to the associated nonautonomous ODE system as a spatially homogeneous base state, we derive sufficient conditions for the existence of a critical chemotactic sensitivity of the first species, denoted by $\chi_{\text{crit}}$. Whenever this threshold is exceeded, the periodic state becomes linearly unstable over a nonempty set of times within each environmental cycle. This gives rise to a recurrent alternation of unstable and stable intervals, during which spatial perturbations are respectively amplified and damped. The analytical derivation is complemented with numerical simulations, in both periodic coexistence and competitive exclusion regimes. Finally, we discuss the temporal location of the instability sets and the possibility of negative instability thresholds.

\end{abstract}

\begin{keyword}
Chemotaxis, ~Pattern Formation,~ Periodicity,~ Turing Instability, ~ Nonautonomous dynamics
\end{keyword}
\end{frontmatter}

\section{Introduction}
Periodic dynamics arise in a wide range of natural systems, particularly in ecological communities involving interacting species. Such periodic behavior has been documented in different phenomena ranging from recurrent forest insect outbreaks \cite{Alps} to cyclic predator--prey dynamics, including the classical lynx-hare interactions \cite{EN42}, arctic rodent populations \cite{GHS03}, and different krill species in marine environments \cite{RQNJOM14, HDE03}. In predator--prey systems, such periodic cycles constitute one of the main features of the classical Lotka--Volterra system, arising in the absence of self-limiting growth effects.
\\\\
Nevertheless, empirical observation of sustained ecological cycles remains mostly limited, since the available data sets typically span only a few number of oscillations due to the long characteristic timescales involved, which may range from several months to decades. Moreover, in laboratory studies, carried out with fast-reproducing species under controlled constant conditions, sustained periodic oscillations seem particularly difficult to reproduce experimentally. In most cases, the results eventually show either stabilization toward a steady state or extinction of one of the species \cite{L73, U57, V79, FEHS00}. From a mathematical perspective, this is consistent with the dynamics of Lotka--Volterra systems incorporating logistic growth terms, with saturation mechanisms effectively suppressing oscillating regimes.
\\\\
This suggests that persistent population cycles (which have only recently been reproduced experimentally under controlled constant conditions \cite{BRWGF20, H20}) are either consequence of more complex mechanisms such as intrinsic delays or density-dependent interactions (see for instance \cite{GKT97} and the references therein), or on the contrary, arise from environmental heterogeneity. In the latter case, particularly for environments with strong seasonality effects, substantially more examples have been documented. For instance, it is well known that seasonal changes in temperature, light intensity, salinity and nutrient availability significantly affect plankton populations \cite{S53}. In particular, this often leads to phytoplankton blooms, i.e. rapid increases in their population levels \cite{SC03}. In turn, blooms of harmful phytoplankton species result in a high release of toxins, consequently inducing mass mortality of numerous fish species \cite{B95}. In such cases, time-periodic dynamics emerge as a consequence of environmental seasonality rather than autonomous species interactions.
\\\\
Similar considerations arise in systems involving competing species, where environmental periodicity may substantially alter the dynamics of the autonomous case, typically balanced between coexistence and competitive exclusion. Namely, for two competing populations with densities $u=u(t)$ and $v= v(t)$, the dynamics of the system
\begin{equation}\label{0.1}
\begin{cases}
   \displaystyle \frac{du}{dt}  = \mu_1 u (1-u-a_1v), & t>0, \\[1.75 ex]
   \displaystyle \frac{dv}{dt}  = \mu_2 v(1-v-a_2u), & t>0,
\end{cases}
\end{equation}
are well-understood (see for instance \cite{Mur}), where $\mu_1, ~\mu_2 >0$ respectively denote the logistic growth rates of species $u$ and $v$, and $a_1, ~a_2 \geq 0$ their interspecific competition coefficients. In particular, if competition is weak, that is, $a_1, a_2 \in [0,1)$, then the coexistence equilibrium 
\begin{equation}\label{0.2}    
(u_*, v_*) :=\left(\frac{1-a_1}{1-a_1a_2}, \frac{1-a_2}{1-a_1a_2}\right),
\end{equation}
is positive and forms a stable node, which attracts all trajectories with arbitrary positive initial data, while the remaining equilibria, $(0,0)$, $(1,0)$ and $(0,1)$ are all unstable. Thus, in this situation, species coexist and their populations densities converge toward the stable equilibrium, without exhibiting sustained oscillatory behavior. If one or both interspecific competition coefficients exceeds $1$, a competitive exclusion result is achieved, with one species dominating the other, which eventually becomes extinct.
\\\\
In contrast, in environments with seasonal variability, the analysis of the following nonautonomous system
\begin{equation}\label{0.3}
\begin{cases}
   \displaystyle \frac{du}{dt}  = u\big(r_1(t)- k_1(t)u-a_1(t) v\big), & t>0, \\[1.75 ex]
   \displaystyle \frac{dv}{dt}  =  v\big(r_2(t)-k_2(t)v-a_2(t)u\big), & t>0,
\end{cases}
\end{equation}
reveals substantial qualitative differences, when coefficients $a_i(t), r_i(t), k_i(t)$, $i \in \{1,2\}$, are assumed to be continuous, positive and periodic with common period $T>0$. Under suitable conditions on the coefficients the existence of a unique positive $T-$periodic solution to system \eqref{0.3} can be proved \cite{AL86}, which moreover attracts any other solution with positive initial data. Under such assumptions, regardless of their initial values, both populations converge to a globally attracting positive periodic coexistence state, whose oscillatory behavior is induced by the cyclic environmental conditions. Alternative hypotheses on the coefficients grant the extinction of one of the species, while the other one persists with $T-$periodic oscillations \cite{A93, TC99}.
\\\\
When spatial dynamics are incorporated into the model, the same distinction between persisting oscillations and convergence to a steady state is preserved when comparing constant environmental conditions with seasonality effects. Namely, let $\Omega \subset \mathbb{R}^n$, $n \geq 1$ be a bounded domain with smooth boundary. Then, for the autonomous parabolic system
\begin{equation}\label{0.6}
    \begin{cases}
        \displaystyle \frac{\partial u}{\partial t} = d_1 \Delta u + \mu_1 u (1-u-a_1v), & x \in \Omega,~ t>0, \\[1.75 ex] 
        \displaystyle \frac{\partial v}{\partial t} = d_2 \Delta v + \mu_2 v (1-v-a_2u), & x \in \Omega,~ t>0,
    \end{cases}
\end{equation}
with diffusion coefficients $d_1, d_2>0$, under homogeneous Neumann boundary conditions, the Lyapunov functional introduced by Goh in \cite{G77} for the multiple-species generalization of the ODE system \eqref{0.1} can be easily adapted to the PDE system \eqref{0.6} \cite{A79, H05}. As a result, regardless of how fast both species diffuse, their long-time behavior is determined by that of the ODE counterpart. In particular, this implies that for weakly competing species, there is a stable coexistence, with all nontrivial nonnegative solutions converging to the coexistence state \eqref{0.2} uniformly in $\Omega$. Similarly, competitive exclusion results hold if one of the coefficients $a_i>1$, leading to the extinction of one of the species \cite{P81, ZP82}. In this way, consistently with the dynamics of the underlying ODE system \eqref{0.1}, the autonomous reaction-diffusion system \eqref{0.6} cannot sustain persistent oscillatory dynamics.
\\\\
In the nonautonomous case, where diffusion is incorporated to system \eqref{0.3}, solutions are marked again by oscillatory dynamics, as a result of the seasonality of the environment. In particular, in \cite{AL89}, the authors showed that the sufficient conditions obtained in \cite{AL86} and in \cite{A93} for the ODE system \eqref{0.3} still determine periodic coexistence or competitive exclusion on its diffusive counterpart, independently of the diffusion coefficients. A further attracting property of the periodic orbit in the case of coexistence was determined in \cite{Rocky}.
\\\\
Therefore, seasonality effects in the environmental conditions ultimately induce cyclic oscillations in the populations ---both in diffusive and nondiffusive settings--- which may arise either through periodic coexistence or through extinction of one species accompanied by periodic persistence of the other. In all these situations, however, the long-time dynamics remain spatially homogeneous and are governed by the balance between the coefficients 
$a_i(t)$, $r_i(t)$ and $k_i(t)$, $i \in \{1,2\}$, which respectively represent the proliferation, self-limitation and competition rates of the species. In particular, for these cases, diffusion preserves the qualitative dynamics of the underlying ODE systems, indicating that spatial spreading alone does not modify the reaction-driven behavior of the species.
\\\\
This motivates the consideration of further migration mechanisms, to investigate whether the attracting states of the previous systems can be destabilized and give rise to spatially heterogeneous oscillatory structures. In particular, incorporating the effect of chemotaxis, i.e. directed migration in response to a concentration gradient, has revealed very rich dynamics in a broad range of systems (see among others \cite{PH11,JWZ16, LSW13}). Moreover, periodicity has been detected in some chemotactically-driven phenomena, such as the movement of \textit{Dictyostelium discoideum} amoebas toward their center of aggregation \cite{SHM91} or in human neutrophils \cite{DZ93}.
\\\\
\textbf{Chemotaxis systems with competing species}
\\\\
To study the effect of a self-generated chemoattractant signal on the behavior of two competing species, Tello and Winkler \cite{TW12} proposed the following Keller--Segel type system,
\begin{equation}\label{0.7}
    \begin{cases}
        \displaystyle \frac{\partial u}{\partial t} = d_1 \Delta u - \chi_1 \nabla \cdot (u \nabla w) + \mu_1 u (1-u-a_1v), & x \in \Omega,~ t>0, \\[1.75 ex] 
        \displaystyle \frac{\partial v}{\partial t} = d_2 \Delta v - \chi_2 \nabla \cdot (v \nabla w) + \mu_2 v (1-v-a_2u), & x \in \Omega,~ t>0, \\[1.75 ex] 
        \displaystyle \tau \frac{\partial w}{\partial t}= \Delta w + \alpha u + \beta v - w , & x \in \Omega,~ t>0, 
    \end{cases}
\end{equation}
under homogeneous Neumann boundary conditions and nonnegative initial data, where $u = u(x,t)$ and $v = v(x,t)$ again represent the population densities of the two competing species, while $w = w(x,t)$ describes the signal concentration. Parameters $\chi_1, \chi_2 >0$ respectively denote the chemotactic sensitivity of each species and $\alpha, \beta >0$ their self-production rate of the substance. Lastly, $\tau \geq 0$ measures the diffusivity of the signal. For $\tau =0$, that is, under a fast diffusion assumption, results in \cite{TW12, STW14} obtained convergence to the corresponding spatially homogeneous steady state under suitable hypotheses in the cases $0 \leq a_i < 1$, $i \in \{1,2\}$ leading to coexistence, and for $0 \leq a_2<1<a_1$ resulting in a competitive exclusion, with $u$ becoming extinct, and $v$ proliferating toward its normalized carrying capacity $1$. To obtain these results, the logistic growth rates $\mu_1$ have to be sufficiently large with respect to the chemotactic sensitivities $\chi_1, \chi_2$, (see also \cite{BLM} for more refined conditions than those in \cite{TW12}) ultimately implying that the possible aggregations due to chemotactic effects are effectively overcome by the reaction terms. The work by Bai and Winkler \cite{BW16} extends these results under similar assumptions to the fully parabolic case, with $\tau>0$, while and Negreanu and Vargas \cite{NV21} considered the nonautonomous setting, similarly yielding convergence of solutions to the underlying ODE reaction dynamics uniformly in space.
\\\\
In contrast, when the dynamics are not governed by large enough values of the logistic growth rates $\mu_i$, different phenomena have been reported to arise. For instance, in absence of kinetic terms, i.e. $\mu_1 = \mu_2 = 0$, finite-time blow-up of solutions to system \eqref{0.7} has been reported for both $\tau = 0$ \cite{BEG13, BG12, ESV09} and $\tau>0$ \cite{LL14}. For positive values of $\mu_i$, when chemotactic sensitivities are allowed to be sufficiently large, different results have characterized Turing instabilities \cite{WYZ17} and the existence of nonconstant steady states \cite{WZYH15}. In particular, when crossing a Hopf bifurcation, periodically heterogeneous solutions could be constructed. 
\\\\
The main purpose of this work is to extend these Turing instability results to the nonautonomous case, considering a system of the form
\begin{equation}\label{0.8}
    \begin{cases}
        \displaystyle \frac{\partial u}{\partial t} = d_1 \Delta u - \chi_1 \nabla \cdot (u \nabla w) + \mu_1 u (1+f_1(t)-u-a_1v), & x \in \Omega,~ t>0, \\[1.75 ex] 
        \displaystyle \frac{\partial v}{\partial t} = d_2 \Delta v - \chi_2 \nabla \cdot (v \nabla w) + \mu_2 v (1+f_2(t)-v-a_2u), & x \in \Omega,~ t>0, \\[1.75 ex] 
        \displaystyle  ~ 0 = \Delta w + \alpha u + \beta v - w , & x \in \Omega, ~t>0, \\[1.75 ex] 
       \displaystyle \frac{\partial u}{\partial \nu} = \frac{\partial v}{\partial \nu} = \frac{\partial w}{\partial \nu} = 0, & x \in \partial \Omega, ~t>0,\\[1.75 ex] 
       u(x,0) =  u_0(x) \geq 0, \quad v(x,0) = v_0(x) \geq 0, & x \in \Omega,
    \end{cases}
\end{equation}
as introduced in \cite{NV21}, where $\nu$ denotes the outward normal vector to $\partial \Omega$. Here $f_1(t), f_2(t)$ are two continuous and bounded $T$-periodic functions satisfying 
\begin{equation}\label{0-hip1}
    1+f_i(t) >0 \quad \text{for all } t >0, ~i \in \{1,2\}.
\end{equation} 
In this way, at each time $t>0$, $1+f_1(t)$ and $1+f_2(t)$ represent the normalized carrying capacity of each species, reflecting the environmental seasonality, which is assumed not to influence neither the logistic growth rates $\mu_i$ nor the competition rates $a_i$, which are considered constant. Moreover, without loss of generality, we may assume that after rescaling the variables, $d_1 = d_2 = 1$.
\\\\
As above mentioned, under suitably small chemotactic sensitivities, namely
\begin{equation}\label{0-nohip}
    \mu_1 > a_2 \mu_2 + 2 \alpha \left( |\chi_1|+|\chi_2| \right), \quad \mu_2 > a_1 \mu_1 + 2 \beta \left( |\chi_1|+|\chi_2| \right),
\end{equation}
for weakly competing species (i.e $a_1, a_2 \in [0,1)$), the dynamics of system \eqref{0.8} are known \cite{NV21} to satisfy
\begin{equation}\label{0-conv}
    \|u(\cdot,t) - \tilde{u}(t)\|_{L^\infty(\Omega)} + \|v(\cdot,t) - \tilde{v}(t)\|_{L^\infty(\Omega)} \to 0 \quad \text{as } t \to \infty,
\end{equation}
where $(\tilde{u},\tilde{v})$ is the unique solution to the ODE system
\begin{equation}\label{0.9}
    \begin{cases}
        \displaystyle \frac{d \tilde{u}}{dt} = \mu_1 \tilde{u}(1+f_1(t) - \tilde{u} - a_1 \tilde{v}), & t>0 \\[1.75 ex]
        \displaystyle \frac{d \tilde{v}}{dt}  = \mu_2 \tilde{v}(1+f_2(t) - \tilde{v} - a_2 \tilde{u}), & t>0,
    \end{cases}
\end{equation}
with initial values
\begin{equation}\label{0.10}
 \tilde{u}(0) = \frac{1}{|\Omega|} \int_\Omega u_0(x) ~dx, \quad \tilde{v}(0) = \frac{1}{|\Omega|} \int_\Omega v_0(x) ~dx.
\end{equation}
However, to the best of our knowledge, the effect of high values of $\chi_i$, where the state $(\tilde{u},\tilde{v})$ could become unstable for system \eqref{0.8}, remains unstudied, as well as other scenarios where one or both of the competition rates exceeds 1.
\\\\
\textbf{Main results}
\\\\
Throughout this work, we assume that the ODE system \eqref{0.9} possesses a unique nonnegative globally attracting $T-$periodic solution. Since every solution of \eqref{0.9} converges to this periodic orbit, the limit in \eqref{0-conv} is independent of the initial values \eqref{0.10}. Accordingly, throughout the remainder of the paper we identify $(\tilde{u}, \tilde{v})$ with this unique $T-$periodic solution.
\\\\
With this globally attracting periodic solution to the ODE system \eqref{0.9} as the reference state, our main objective is to characterize its transient linear instability within the full PDE system \eqref{0.8}. To this end, we fix all parameters except for $\chi_1$, which we consider as the bifurcation parameter, with the goal of obtaining recurrent unstable dynamics marked by the periodicity of $f_1$ and $f_2$. In particular, we identify time intervals in which $(\tilde{u},\tilde{v})$ becomes linearly unstable, where perturbations are amplified due to the strong taxis effect, followed by stable intervals, where these perturbations decay, and population levels are brought back to $(\tilde{u},\tilde{v})$. Over time, this mechanism gives rise to recurrent aggregation patterns synchronized with the environmental periodicity.
\\\\
This contrasts with classical Turing instabilities in autonomous systems, where unstable modes remain permanently unstable once the instability threshold is crossed, or with transient growth phenomena in the autonomous Keller-Segel system with logistic growth in which, if diffusion is small enough, during an intermediate transient phase, population levels can achieve arbitrarily high values, exceeding the environmental carrying capacity, only to decay afterwards, ultimately retaining low densities \cite{Wink-trans, Lank-trans}. In our case, for system \eqref{0.8} transient aggregation phenomena within the interval $[0,T]$ may become recurrent without the presence of a Hopf bifurcation, driven solely by the environmental periodicity.
\\\\ 
The instability conditions are described in the following theorem, which yields a threshold value, denoted $\chi_{\text{crit}}$ such that for any $\chi_1 >\chi_{\text{crit}}$, at least one Fourier mode becomes unstable over a certain instability set within $[0,T]$.
\begin{theorem}\label{t1}
    Assume that $f_1(t)$, $f_2(t) \in  C^1([0,T])$ are $T-$periodic functions satisfying \eqref{0-hip1} and that $\big(\tilde{u}(t), \tilde{v}(t)\big)$ is a nonnegative globally attracting $T-$periodic solution of the ODE system \eqref{0.9} such that
    \begin{equation}\label{0-hip2}
    \alpha \big(1+f_2(t)\big)\leq \alpha a_2 \tilde{u}(t) + (2\alpha - a_2 \beta) \tilde{v}(t), \quad \text{for all } t \in [0,T].        
    \end{equation}      
    Then, there exists $\chi_{\text{crit}} \in \mathbb{R}$ such that if $\chi_1 > \chi_{\text{crit}}$, there is a nonempty set of times  $\mathcal{I} \subset [0,T]$ over which the associated spatially homogeneous state $(\tilde{u}, \tilde{v}, \tilde{w})$, where $\tilde{w} := \alpha \tilde{u} + \beta \tilde{v}$, is linearly unstable for system \eqref{0.8}.
\end{theorem}
To derive the result, we adapt the method introduced in \cite{VG20} for nonautonomous reaction--diffusion systems to the cross--diffusive structure of the chemotaxis system \eqref{0.8}. This appears to be the first contribution concerning nonautonomous Turing instabilities in chemotaxis--type systems. Theorem \ref{t1} is complemented by different numerical simulations, in which the value of $\chi_{\text{crit}}$ is approximated in order to show the destabilizing effect of a large enough $\chi_1$ and the subsequent recurrent pattern formation process.
\\\\
The role of hypothesis \eqref{0-hip2} in establishing Theorem \ref{t1} is to provide a solvability criterion for the values of $\chi_c^{(k)}$ and $\widehat{\chi}_c^{~(k)}$, which are used to define $\chi_{\text{crit}}$. This becomes clear in the proof of Lemma \ref{l2} and Lemma \ref{l2-2}. We remark that assumption \eqref{0-hip2} is only a sufficient condition, however, it is easy to check numerically once the base state $\big(\tilde{u}(t), \tilde{v}(t) \big)$ has been computed. A further but more restrictive condition ensuring that \eqref{0-hip2} is fulfilled can also be derived, based only on the parameter values, as described in Remark \ref{r-ult}.
\\\\
Following this introduction, the article is structured as follows. In Section \ref{s2}, we review some relevant results concerning the ODE system \eqref{0.9}, understood within a broader class of periodic competitive systems \eqref{0.3}. In particular, sufficient conditions for the existence of a unique globally attracting state are collected in \eqref{1.5} and \eqref{1.6}, corresponding to coexistence and competitive exclusion, respectively. These results provide base states $\big(\tilde{u}(t), \tilde{v}(t)\big)$, which are globally attracting solutions to the ODE system \eqref{0.9}. The linear instability conditions for $\big(\tilde{u}(t), \tilde{v}(t)\big)$ are then derived in Section \ref{s4} through a sequence of auxiliary lemmas, which allow us to to prove Theorem \ref{t1}. Next, Section \ref{s5} includes different numerical simulations concerning both cases: first when $\big(\tilde{u}(t), \tilde{v}(t)\big)$ denotes a periodic coexistence state, and subsequently in a competitive exclusion scenario, where $\tilde{v}(t) \equiv 0$. Some remarks concerning the location of the instability sets and the case of negative chemotactic sensitivities are gathered in Section \ref{s6}, and lastly, a discussion of the results obtained and the conclusions are presented in Section \ref{s7}.

\section{Preliminaries: Extinction and periodic coexistence in nonautonomous competitive dynamics}\label{s2}
To begin the linear instability analysis of system \eqref{0.8}, we first review some known results concerning the dynamics of the associated ODE system \eqref{0.9}, not accounting for diffusive and chemotactic effects. System \eqref{0.9} arises as a particular case of the more general class of nonautonomous competitive systems \eqref{0.3}, which we reproduce here for convenience
\begin{equation}\label{1.0}
\begin{cases}
   \displaystyle \frac{du}{dt}  = u\big(r_1(t)- k_1(t)u-a_1(t) v\big), & t>0, \\[1.75 ex]
   \displaystyle \frac{dv}{dt}  =  v\big(r_2(t)-k_2(t)v-a_2(t)u\big), & t>0.
\end{cases}
\end{equation}
As mentioned in the Introduction, in contrast to the autonomous setting, time-periodic coefficients may lead to attracting periodic solutions. The first results in this direction are due to de Mottoni and Schiaffino \cite{MSch}. In particular, by assuming that functions $r_i$, $k_i$ and $a_i$ are positive, continuous $T-$periodic, for $i \in \{1,2\}$, they showed that any nonnegative solution to system \eqref{1.0} converges component-wise to another $T-$periodic solution of \eqref{1.0}. As a result, the dynamics of the system are asymptotically determined by the three different types of nonnegative $T-$periodic solutions of \eqref{1.0}, namely the trivial state $(0,0)$; the semitrivial states, of the form $(u,0)$ and $(0,v)$; and coexistence states, constituted by positive $T-$periodic solutions. Different results have been obtained over the years concerning stability of the different states, as well as existence, multiplicity and attractivity properties of coexistence solutions.
\\\\
The work of, among others, Ahmad, Álvarez and Lazer in the 1980s and 1990s led to different sufficient conditions for the existence of a unique globally attracting coexistence state, as well as for competitive exclusion behaviors. In particular, by introducing the notation
$$
f^M = \max_{t \in [0,T]} f(t), \quad f^L = \min_{t \in [0,T]} f(t), 
$$
for an arbitrary bounded function $f$, the following results were obtained
\begin{enumerate}[(i)]
    \item \label{i1} In \cite{AL86}, it was shown that the assumptions
    \begin{equation}\label{1.1}
         r_1^L  > \frac{a_1^M r_2^M}{k_2^L}, \quad r_2^L > \frac{a_2^M r_1^M}{k_1^L},
    \end{equation}
    imply the existence of a unique positive $T-$periodic solution to system \eqref{1.0}, $(u^{\text{per}}(t),v^{\text{per}}(t))$ which moreover attracts any other solution with positive initial data. In addition, $(u^{\text{per}}(t),v^{\text{per}}(t))$ is such that
    \begin{equation}\label{1.2}
    \begin{split}
        \frac{r_1^L k_2^L - a_1^M r_2^M}{k_1^M k_2^L -  a_1^M a_2^L} \leq u^{\text{per}}(t) &\leq \frac{r_1^M k_2^M - a_1^Lr_2^L}{k_1^L k_2^M -a_1^La_2^M},  \\[1.5 ex]
        \frac{k_1^L r_2^L - r_1^M a_2^M}{k_1^L k_2^M - a_1^L a_2^M} \leq v^{\text{per}}(t)& \leq \frac{k_1^M r_2^M - a_2^L r_1^L}{k_1^M k_2^L - a_2^L a_1^M},
    \end{split}
    \end{equation}    
    This was further extended in \cite{A87} to the case of general nonperiodic coefficients, only assuming positivity, continuity and boundedness of $r_i$, $k_i$ and $a_i$. 
    \item \label{i2} In \cite{A93}, condition \eqref{1.1} is replaced by
    \begin{equation}\label{1.3}
         r_1^L  > \frac{a_1^M r_2^M}{k_2^L}, \quad r_2^M \leq \frac{r_1^La_2^L }{k_1^M},
    \end{equation}
    in which the first inequality coincides with that of \eqref{1.1}, while in the second one, maxima and minima are interchanged, as well as the direction of the inequality. Under this assumption, a competitive exclusion is obtained, with the second species eventually becoming extinct, namely
    \begin{equation}\label{1.4}
        v(t) \to 0 \quad \text{and} \quad u(t) - u^{\text{per}}(t) \to 0 \quad \text{as } t \to \infty,
    \end{equation}
    where $u^{\text{per}}(t)$ is the unique periodic solution to the decoupled logistic equation
    $$
    u' = u(t)  \big(r_1(t) - k_1(t) u\big), \quad t>0.
    $$
    Naturally, a similar result is obtained if the roles of $u$ and $v$ are interchanged.
\end{enumerate}
Thus, to guarantee a unique globally attracting periodic coexistence state, conditions \eqref{1.1} require the minimum intrinsic growth rates $r_i^L$ and self-limitation effects $k_i^L$ of the species to be sufficiently strong, compared to the maximal interspecific competition rates $a_i^M$. In contrast, if \eqref{1.1} fails and \eqref{1.3} is satisfied, then the dominant species persists, reaching a periodic regime, while the other one becomes extinct.
\\\\
With respect to the particular case of system \eqref{0.9}, we note that (\ref{i1}) and (\ref{i2}) read: 
\begin{enumerate}
    \item If 
    \begin{equation}\label{1.5}
        1+f_1^L > a_1\big(1+f_2^M\big), \quad 1+f_2^L > a_2\big(1+f_1^M\big),
    \end{equation}
    then both species coexist periodically. In particular, the system admits a unique positive, periodic and globally attracting coexistence state $\big(u^\text{per}(t), v^{\text{per}}(t)\big)$, to which all other solutions converge to.
    \item If
    \begin{equation}\label{1.6}
        1+f_1^L > a_1\big(1+f_2^M\big), \quad 1+f_2^M \leq a_2\big(1+f_1^L\big),
    \end{equation}
    then the second species $v$ becomes extinct, while the first one dominates, and persists with periodic oscillations, converging to $u^\text{per}(t)$, the unique periodic solution to 
    $$
    u' =   \mu_1 u \big(1+f_1(t) - u\big).
    $$
    As a result, this time the orbit $\big(u^\text{per}(t), 0\big)$ is the only global attractor of the system.
\end{enumerate}
We remark that these are only sufficient conditions, and shall not be interpreted as sharp thresholds. A characterization of the asymptotic stability of the semitrivial states is due to Cushing \cite{Cushing}. In particular, by assuming that
\begin{equation}\label{1-a0}
\int_0^T r_1(t) \, dt >0, \quad \int_0^T r_2(t)\, dt >0
\end{equation}
one can define $\hat{u}$ and $\hat{v}$ as the unique periodic solutions to the uncoupled logistic equations
\begin{equation}\label{1-a1}
\frac{d \hat{u}}{dt} = \hat{u} \big( r_1(t) - k_1(t) \hat{u} \big), \quad 
        \frac{d \hat{v}}{dt} = \hat{v} \big( r_2(t) - k_2(t) \hat{v} \big). \\
\end{equation}
Then, by introducing 
\begin{equation}\label{1-a2}
    \lambda_1 = \frac{1}{T}\int_0^T\big( k_1(t) \hat{u}(t)  -a_1(t) \hat{v}(t) \big) \, dt, \quad \lambda_2 = \frac{1}{T} \int_0^T \big( k_2(t) \hat{v}(t) -  a_2(t) \hat{u}(t) \big) \, dt 
\end{equation}
one has that the semitrivial state $(\hat{u},0)$ is linearly unstable if and only if $\lambda_2>0$, and linearly asymptotically stable if and only if $\lambda_2< 0$. Similarly, the remaining semitrivial state $(0,\hat{v})$ is linearly stable if and only if $\lambda_1 > 0$
and linearly asymptotically stable if and only if $\lambda_1 < 0$. 
\\\\
If both $\lambda_1,\lambda_2>0$, then both semitrivial states are unstable and the system is persistent. Moreover, this implies that the system admits at least one stable coexistence state (see Eilbeck and López-Gómez \cite{ELG} for this and for other multiplicity results).
\\\\
For this work, we only consider states that are global attractors of the ODE system \eqref{0.9}, to ensure that instability arises as a result of the spatial effects, and more precisely, of chemotaxis.

\section{Linear instability and recurrent aggregation}\label{s4}

Having described the sufficient conditions ensuring the existence of a unique nonnegative globally attracting $T-$periodic solution to the ODE system \eqref{0.9} ---either corresponding to periodic coexistence or to competitive exclusion---, we now turn to the analysis of its linear stability when regarded as a spatially homogeneous state of the PDE system \eqref{0.8}. Following the notation introduced in the Introduction, throughout the remainder of this work, we denote this periodic solution by $(\tilde{u}, \tilde{v})$.
\\\\
To investigate the linear stability of the spatially homogeneous state
$(\tilde{u},\tilde{v}, \tilde{w})$ for the PDE system \eqref{0.8}, where $\tilde{w}(t) := \alpha \tilde{u}(t) + \beta \tilde{v}(t)$, we consider perturbed solutions of the form
\begin{equation}\label{4.1}
    u(x,t) = \tilde{u}(t) + \varepsilon U(x,t), \quad v(x,t) = \tilde{v}(t) + \varepsilon V(x,t), \quad w(x,t) = \tilde{w}(t) + \varepsilon W(x,t),
\end{equation}
where $0 \leq \varepsilon \ll 1$. Substituting solutions \eqref{4.1} into system \eqref{0.8} and neglecting the terms of order $\varepsilon^2$ and above, we arrive at the following nonautonomous linearized system for $(U,V,W)$, given by
\begin{equation}\label{4.2}
    \begin{cases}
        U_t = \Delta U - \chi_1 \tilde{u}(t) \Delta W + \mu_1\big(1+f_1(t) - 2 \tilde{u}(t) - a_1\tilde{v}(t)\big) U - \mu_1 a_1 \tilde{u}(t) V, & x \in \Omega, ~t>0, \\[1.5 ex]
        V_t = \Delta V - \chi_2 \tilde{v}(t) \Delta W - \mu_2 a_2 \tilde{v}(t) U + \mu_2 \big(1+f_2(t) - 2\tilde{v}(t) - a_2 \tilde{u}(t) \big) V, & x \in \Omega, ~t>0, \\[1.5 ex]
         \hspace{0.15 cm} 0 = \Delta W + \alpha U + \beta V - W, & x \in \Omega, ~t>0,
    \end{cases}
\end{equation}
together with homogeneous Neumann boundary conditions and the respective initial values for $U$ and $V$. By considering the set of eigenvalues $\{\lambda_k\}_{k\in\mathbb{N}_0}$ and eigenfunctions $\{\varphi_k\}_{k\in\mathbb{N}_0}$ of the operator $-\Delta$ with homogeneous Neumann boundary conditions, we can expand $(U,V,W)$ in Fourier series as follows
\begin{equation}\label{4.3}
 U(x,t) = \sum_{k=0}^\infty
        C_1^{(k)}(t) \varphi_k (x), \quad   V(x,t) = \sum_{k=0}^\infty
        C_2^{(k)}(t) \varphi_k (x), \quad  W(x,t) = \sum_{k=0}^\infty
        C_3^{(k)}(t) \varphi_k (x),
\end{equation}
where coefficients $\big( C_1^{(k)}(t),  C_2^{(k)}(t), C_3^{(k)}(t) \big)$ represent the amplitude of the $k-$th spatial mode of the perturbation at time $t>0$.
\\\\
The elliptic nature of the third equation allows for further simplification. In particular, by substituting the Fourier expansions \eqref{4.3}, and noting that $\Delta \varphi_k = - \lambda_k \varphi_k$ for each $k \in \mathbb{N}_0$, one can solve for coefficients $C_3^{(k)}(t)$, yielding
\begin{equation}\label{4.4}
    C_3^{(k)}(t) = \frac{\alpha  C_1^{(k)}(t) + \beta  C_2^{(k)}(t) }{1+ \lambda_k}, \quad t>0.
\end{equation}
Proceeding similarly in the first two equations, after substituting the expression for $C_3^{(k)}$ obtained in \eqref{4.4}, for each $k \in \mathbb{N}_0$ we arrive at a linear nonautonomous ODE system for $C_1^{(k)}(t)$ and $C_2^{(k)}(t)$ of the form
\begin{equation}\label{4.5}
\frac{d}{dt}
    \begin{pmatrix}
        C_1^{(k)} \\
         C_2^{(k)}
    \end{pmatrix} = A_k(t) \begin{pmatrix}
        C_1^{(k)} \\
         C_2^{(k)}
    \end{pmatrix}, \quad t>0
\end{equation}
where for each $k \in \mathbb{N}_0$, the coefficient matrix $A_k(t)$ is given by
\begin{equation*}
    A_k(t) = \begin{pmatrix}
    A_{11}^{(k)}(t) & A_{12}^{(k)}(t) \\
     A_{21}^{(k)}(t) & A_{22}^{(k)}(t),
    \end{pmatrix}
\end{equation*}
with
\begin{equation}\label{4.6}
    \begin{split}
        A_{11}^{(k)}(t) & :=  \displaystyle -\lambda_k + \chi_1 \tilde{u}(t) \alpha \frac{\lambda_k}{1+\lambda_k} + \mu_1 \big(1+f_1(t) - 2 \tilde{u}(t) - a_1 \tilde{v}(t) \big), \\[1.5 ex]
        A_{12}^{(k)}(t) & :=  \displaystyle \chi_1 \beta \hspace{0.03 cm} \tilde{u}(t) \frac{\lambda_k}{1+\lambda_k}  - \mu_1 a_1 \tilde{u}(t),     \\[1.5 ex]
        A_{21}^{(k)}(t) & := \displaystyle \chi_2 \hspace{0.03 cm}\alpha \hspace{0.03 cm}\tilde{v}(t) \frac{\lambda_k}{1+\lambda_k} - \mu_2 a_2 \tilde{v}(t) , \\[1.5 ex]
         A_{22}^{(k)}(t) & :=  \displaystyle -\lambda_k + \chi_2 \beta \hspace{0.03 cm} \tilde{v}(t)  \frac{\lambda_k}{1+\lambda_k} + \mu_2 \big(1+f_2(t) - 2 \tilde{v}(t) - a_2 \tilde{u}(t) \big),        
    \end{split}
\end{equation}
The main difference with respect to classical autonomous Turing instabilities lies in the time dependence of the matrix $A_k(t)$. In the autonomous setting, the growth of the $k-$th mode is completely determined by the sign of the eigenvalues of a constant matrix. In contrast, the nonconstant structure of \eqref{4.5} prevents such a direct characterization. 
\\\\
While one alternative is to study the Floquet multipliers of $A_k$ for each $k \in \mathbb{N}_0$, that provides a criterion for global instability. Instead, to detect transient instability over subintervals of $[0,T]$, we analyze the instantaneous growth of $C_1^{(k)}$ and $C_2^{(k)}$. In particular, by differentiation each equation on system \eqref{4.5}, it can be converted into a pair of uncoupled second-order linear ODEs of the form
\begin{equation}\label{4.7}
\begin{cases}
    \left(C_1^{(k)} \right)'' + P_k(t) \left(C_1^{(k)} \right)' + Q_k(t) \hspace{0.05 cm} C_1^{(k)} = 0, & t>0, \\[2 ex]
    \left(C_2^{(k)} \right)'' + \widehat{P}_k(t) \left(C_2^{(k)} \right)' + \widehat{Q}_k(t) \hspace{0.05 cm} C_2^{(k)} = 0, & t>0, 
\end{cases}
\end{equation}
with $' = d/dt$. From now on, we distinguish two different cases: one where $\big(\tilde{u}(t), \tilde{v}(t)\big)$ represents a coexistence state, with $\tilde{u}(t), \tilde{v}(t) >0$ for all $t>0$, and secondly the semitrivial state $\big(\tilde{u}(t), \tilde{v}(t)\big) = \big(\tilde{u}(t), 0\big)$, due to the favorable cancellations in the latter case. Naturally, the other semitrivial state $\big(0, \tilde{v}(t)\big)$ follows similarly by exchanging the variables.
\\\\
In the coexistence case, coefficients $P_k$, $Q_k$, $\widehat{P}_k$ and $\widehat{Q}_k$ in \eqref{4.7} are given by
\begin{equation}\label{4.8}
\begin{aligned}
     & P_k(t) := -\tr\big(A_k(t)\big) - \frac{\big(A_{12}^{(k)}(t) \big)'}{A_{12}^{(k)}(t)},  & Q_k(t):=  \det\big(A_k(t)\big) - \big(A_{11}^{(k)}(t)\big)'  + A_{11}^{(k)}(t)  \frac{\big(A_{12}^{(k)}(t) \big)'}{A_{12}^{(k)}(t)}, \\[2 ex]
        &\widehat{P}_k(t) :=   -\tr\big(A_k(t)\big) - \frac{\big(A_{21}^{(k)}(t) \big)'}{A_{21}^{(k)}(t)}, &  \widehat{Q}_k(t):=  \det\big(A_k(t)\big) - \big(A_{22}^{(k)}(t)\big)'  + A_{22}^{(k)}(t)  \frac{\big(A_{21}^{(k)}(t) \big)'}{A_{21}^{(k)}(t)},
\end{aligned}
\end{equation}
for all $t>0$, while for the semitrivial state $\big(\tilde{u}(t), 0\big)$, we have $A_{21}^{(k)}(t) \equiv 0$ for all $t >0$ and all $k \in \mathbb{N}^+$. In particular, the equation for $C_2^{(k)}$ is uncoupled, and the coefficients are given by
\begin{equation}\label{4.8-2}
\begin{aligned}
   P_k(t)          &:= -\tr\big(A_k(t)\big) - \frac{\big(A_{12}^{(k)}(t) \big)'}{A_{12}^{(k)}(t)},
&
Q_k(t)          &:= \det\big(A_k(t)\big) - \big(A_{11}^{(k)}(t)\big)'  + A_{11}^{(k)}(t)  \frac{\big(A_{12}^{(k)}(t) \big)'}{A_{12}^{(k)}(t)},
\\[2ex]
\widehat{P}_k(t)&:= 0,
&
\widehat{Q}_k(t)&:= -A_{22}^{(k)}(t).
\end{aligned}
\end{equation}
In both cases, although a complete characterization of the instantaneous growth of $C_1^{(k)}$ and $C_2^{(k)}$ may depend on the specific form of the coefficients, a simple sufficient criterion can be found in Theorem A.1 in \cite{VG20}, which we reproduce here for completeness.
\begin{lemma}[Theorem A.1 in \cite{VG20}]\label{l1-aux}
Consider linear second-order ODE of the form
$$
y'' + \mathcal{P}(t) y' + \mathcal{Q}(t) y = 0, \quad t>0,
$$
and assume that there exists an interval $\mathcal{I} \subset [0,\infty)$ such that $\mathcal{Q}(t)<0$ for all $t \in \mathcal{I}$. Then, for any nonzero choice of initial data, the solution $y(t)$ has an exponential growth rate on $\mathcal{I}$.    
\end{lemma}
Thus, from the expressions in \eqref{4.8} and \eqref{4.8-2}, for each fixed $t \in [0,T]$ we can interpret $Q_k(t)$ and $\widehat{Q}_k(t)$ as functions of the bifurcation parameter $\chi_1$, which we respectively denote by $ \mathcal{Q}_k(\chi_1,t)$ and $\widehat{\mathcal{Q}}_k(\chi_1,t)$. In particular, we first prove that under suitable conditions, the equations
\begin{equation}\label{4-eq}
\mathcal{Q}_k(\chi_1,t) = 0, \quad \text{and}\quad \widehat{\mathcal{Q}}_k(\chi_1,t) = 0,    
\end{equation}
can be uniquely solved with respect to $\chi_1$. This will allow us  to identify threshold values associated with the transient amplification of the $k$-th perturbation mode. To guarantee the well-posedness of these thresholds, we first establish the following technical lemma concerning the solvability of \eqref{4-eq}. We begin by the first case, corresponding to the coexistence state, in which $\mathcal{Q}_k(\chi_1,t) = 0$ and $\widehat{\mathcal{Q}}_k(\chi_1,t) = 0$ are as given by the function in \eqref{4.8}.
\begin{lemma}\label{l2}
    Let $f_1(t)$, $f_2(t) \in  C^1([0,T])$ be $T-$periodic functions satisfying \eqref{0-hip1} and assume that $(\tilde{u}, \tilde{v})$ is a positive, globally attracting $T-$periodic coexistence state of system \eqref{0.9} such that \eqref{0-hip2} holds. Then for each $k \in \mathbb{N}^+$, there exist continuous functions $a_k, b_k, \widehat{b}_k: [0,T] \to \mathbb{R} $ such that
    \begin{equation}\label{4-cont}
        \mathcal{Q}_k(\chi_1,t)=a_k(t)\chi_1+b_k(t), \quad \text{and} \quad   \widehat{\mathcal{Q}}_k(\chi_1,t)=a_k(t)\chi_1+\widehat{b}_k(t)
    \end{equation}  
    with the property that $a_k(t)<0$ for all $t \in [0,T]$. Moreover, for each $t \in [0,T]$, equations $\mathcal{Q}_k(\chi_1,t) = 0$ and $\widehat{\mathcal{Q}}_k(\chi_1,t) = 0$ admit a unique solution in $\chi_1$, denoted respectively by $ \chi_1^{(k)}(t)$ and $\widehat{\chi}_1^{~(k)}(t)$.
\end{lemma}
\begin{proof}\hfill\break
\hspace*{1em}
For a fixed $k \in \mathbb{N}^+$ we first prove that for each $t \in [0,T]$, the maps $
\chi_1 \mapsto \mathcal{Q}_k(\chi_1,t)$ and $
\chi_1 \mapsto \widehat{\mathcal{Q}}_k(\chi_1,t)$ are affine. Indeed, we recall that from \eqref{4.8} one has
\begin{equation*}
    \begin{split}
        Q_k(t)= \det\big(A_k(t)\big) - \big(A_{11}^{(k)}(t)\big)'  
+ A_{11}^{(k)}(t)  
\frac{\big(A_{12}^{(k)}(t) \big)'}{A_{12}^{(k)}(t)}, \\[1.5 ex]
\widehat{Q}_k(t)=  \det\big(A_k(t)\big) - \big(A_{22}^{(k)}(t)\big)'  + A_{22}^{(k)}(t)  \frac{\big(A_{21}^{(k)}(t) \big)'}{A_{21}^{(k)}(t)}, 
    \end{split}
\end{equation*}
where we shall see that the only remaining contribution depending on $\chi_1$ comes from $ \det\big(A_k(t)\big)$, and that moreover this dependence is linear. First, since $\tilde{u}$ satisfies the first equation in \eqref{0.9}, one has that
\begin{equation}\label{4-int}
    \begin{split}
        \frac{\big(A_{12}^{(k)}(t) \big)'}{A_{12}^{(k)}(t)}
&=
\mu_1 \big(1+f_1(t)-\tilde{u}(t)-a_1\tilde{v}(t)\big), \\[1.5 ex]
\big(A_{11}^{(k)}(t) \big)' & = \chi_1 \alpha \frac{\lambda_k}{1+\lambda_k} \mu_1 \tilde{u}(t) \big(1+f_1(t) - \tilde{u}(t) - a_1 \tilde{v}(t) \big) + C(t),
    \end{split}
\end{equation}
for all $t \in [0,T]$, where 
$$
C(t) := \mu_1 \Big[f_1'(t) - 2 \mu_1 \tilde{u}(t) \big(1+f_1(t) - \tilde{u}(t) - a_1 \tilde{v}(t) \big) - a_1 \mu_2 \tilde{v}(t) \big(1+f_2(t)-\tilde{v}-a_2 \tilde{u}\big) \Big],
$$
which is independent of $\chi_1$. Furthermore, the only term depending on $\chi_1$ in $A_{11}^{(k)}$ is precisely of the form $\chi_1 \tilde{u}(t) \alpha \frac{\lambda_k}{1+\lambda_k}$. Thus, it is direct to check that from \eqref{4-int}, 
$$- \big(A_{11}^{(k)}(t)\big)'  
+ A_{11}^{(k)}(t)  
\frac{\big(A_{12}^{(k)}(t) \big)'}{A_{12}^{(k)}(t)} \quad \text{is independent of } \chi_1, $$
for all $t\in [0,T]$. As a result, the only term in $\mathcal{Q}_k(\chi_1,t)$ depending on $\chi_1$ is $\det\big(A_k(t)\big)$, and the dependence is linear, as only $A_{11}^{(k)}$ and $A_{12}^{(k)}$ depend on $\chi_1$. 
\\\\
Similarly, for $\widehat{\mathcal{Q}}_k(\chi_1,t)$, as $A_{21}^{(k)}$ and $A_{22}^{(k)}$ are independent of $\chi_1$, it follows that
$$- \big(A_{22}^{(k)}(t)\big)'  + A_{22}^{(k)}(t)  \frac{\big(A_{21}^{(k)}(t) \big)'}{A_{21}^{(k)}(t)} \quad \text{is independent of } \chi_1, $$
for all $t \in [0,T]$. As a result, \eqref{4-cont} holds, that is, for each fixed $t \in [0,T]$, there exist three continuous functions $a_k,b_k, \widehat{b}_k : [0,T]\to \mathbb{R}$ such that
$$
    \mathcal{Q}_k(\chi_1,t)=a_k(t)\chi_1+b_k(t), \quad \text{and} \quad   \widehat{\mathcal{Q}}_k(\chi_1,t)=a_k(t)\chi_1+\widehat{b}_k(t).
$$
We note that the coefficient $a_k(t)$ of $\chi_1$ is the same for $\mathcal{Q}_k$ and for $\widehat{\mathcal{Q}}_k$, as it comes from $\det\big(A_k(t)\big)$. 
\\\\
We lastly show that under hypothesis \eqref{0-hip2}, we have that $a_k(t)<0$ on $[0,T]$, and as a result, equations $\mathcal{Q}_k(\chi_1,t) = 0$ and $\widehat{\mathcal{Q}}_k(\chi_1,t) = 0$ can be uniquely solved with respect to $\chi_1$. In particular, after some computations, one arrives at
\begin{equation} \label{a-k-eq}
    a_k(t) = \tilde{u}(t) \frac{\alpha \lambda_k}{1+\lambda_k} \left(- \lambda_k +  \mu_2 \left[ 1+f_2(t) - a_2 \tilde{u}(t)  - \left(2 - \frac{a_2\beta}{\alpha} \right) \tilde{v}(t)\right]\right), \quad \text{for all } t \in [0,T].
\end{equation}
To finish the proof, as a consequence of \eqref{0-hip2} we have that
$$
1+f_2(t) - a_2 \tilde{u}(t)  - \left(2 - \frac{a_2\beta}{\alpha} \right) \tilde{v}(t) \leq 0,  \quad \text{for all } t \in [0,T],
$$
and the positivity of $\tilde{u}(t),\alpha,\lambda_k$ and $\mu_2$ implies that $a_k(t)<0$ over $[0,T]$. Thus, both equations have a unique solution, given by
\begin{equation}\label{chis-def}
    \chi_1^{(k)}(t) := - \frac{b_k(t)}{a_k(t)}, \quad \widehat{\chi}_1^{(k)} := - \frac{\widehat{b}_k(t)}{a_k(t)}, \quad \text{for all } t \in [0,T],
\end{equation}
completing the proof.
\end{proof}
Next, we prove a similar result for the semitrivial state $\big(\tilde{u}(t),0\big)$. This time, however, we shall see that the cancellations in the second equation result in only being able to define $\chi_1^{(k)}$, as for the second equation one always has $\widehat{\mathcal{Q}}_k(\chi_1,t) >0$ for all $k \in \mathbb{N}^+$ and all $t >0$, $\chi_1 \in \mathbb{R}$.

\begin{lemma}\label{l2-2}
    Assume that functions $f_1(t)$, $f_2(t) \in  C^1([0,T])$ are $T-$periodic and that they fulfill \eqref{0-hip1}. Suppose moreover that $\big(\tilde{u}(t), \tilde{v}(t)\big) = \big(\tilde{u}(t),0 \big)$ is a $T-$periodic globally attracting semitrivial state of system \eqref{0.9} such that the corresponding version of \eqref{0-hip2} is satisfied. Then, for each $k \in \mathbb{N}^+$, there exist continuous functions $a_k, b_k: [0,T] \to \mathbb{R} $ such that
    \begin{equation}\label{4-cont-2}
        \mathcal{Q}_k(\chi_1,t)=a_k(t)\chi_1+b_k(t),
    \end{equation}  
    with $a_k(t)<0$ for all $t \in [0,T]$. As a result, for each $t \in [0,T]$, the equation $\mathcal{Q}_k(\chi_1,t) = 0$ has a unique solution in $\chi_1$, denoted by $\chi_1^{(k)}(t)$. Moreover, $\widehat{\mathcal{Q}}_k(\chi_1,t)>0$ for all $k \in \mathbb{N}^+$, all $t>0$ and all $\chi_1 \in \mathbb{R}$.
\end{lemma}
\begin{proof}\hfill\break
\hspace*{1em}
The only difference with respect to the previous coexistence case lies on the second equation, as now $A_{21}^{(k)} (t) \equiv 0$ for all $t>0$ and all $k \in \mathbb{N}^+$. Consequently, according to \eqref{4.8-2}, $C_2^{(k)}$ now satisfies the second-order equation
$$
\Big(C_2^{(k)}\Big)'' - A_{22}^{(k)}(t)  \,C_2^{(k)} = 0, \quad t>0.
$$
In particular, as $\tilde{v}(t) \equiv 0$, the expression for $A_{22}^{(k)}(t)$ in \eqref{4.6} now simplifies to
$$
A_{22}^{(k)} (t) = \displaystyle -\lambda_k  + \mu_2 \big(1+f_2(t) - a_2 \tilde{u}(t) \big),    \quad \text{for all } t >0.    
$$
Moreover, by hypothesis \eqref{0-hip2}, again since $\tilde{v}(t) \equiv 0$, we have that
$$
1+f_2(t) \leq a_2 \tilde{u}(t), \quad \text{for all } t >0,
$$
which directly implies that $A_{22}^{(k)} (t) < 0 $. As a result, we have that for all $k \in \mathbb{N}^+$ and for any choice of $\chi_1 \in \mathbb{R}$
$$
\widehat{\mathcal{Q}}_k(\chi_1,t) =- A_{22}^{(k)} (t) > 0, \quad \text{for all } t>0.
$$
For $\mathcal{Q}_k(\chi_1,t)$, as $A_{12}^{(k)}(t)$ does not in general vanish, the result follows from the same steps as in the proof of Lemma \ref{l2}.
\end{proof}
As a result of Lemma \ref{l2-2}, in the case of a semitrivial state of the form $\big(\tilde{u}(t), 0 \big)$, the transient instability condition from Lemma \ref{l1-aux} can never be satisfied by the modal amplitudes $C_2^{(k)}$, $k \in \mathbb{N}^+$. Hence, no instability of these modes can be detected by means of this criterion. We remark however that since Lemma \ref{l1-aux} is only a sufficient condition for exponential growth, this does not imply that the corresponding perturbations are stable or decay to zero. 
\begin{remark}\label{r-ult}
We emphasize again that condition \eqref{0-hip1} is not implicit, and can be numerically checked with the selected values of $\alpha, \beta, a_2 >0$ once the base state $(\tilde u,\tilde v)$ has been computed. Moreover, using the bounds in \eqref{1.2}, a further sufficient condition can be derived involving only the parameters of the system, without requiring explicit knowledge of $(\tilde u,\tilde v)$. Nevertheless, it is generally more restrictive, as it replaces the pointwise verification of \eqref{0-hip2} by a global estimate.
\end{remark}

Guaranteeing the solvability of equations \eqref{4-eq}, or for the first of them in the case of a semitrivial state, allows us to define thresholds values for $\chi_1$, which characterize the growth of at least one Fourier mode of the perturbations throughout a certain subinterval of $[0,T]$. As previously, we begin by the case in which $\big(\tilde{u}(t), \tilde{v}(t)\big)$ represents a positive coexistence state.

\begin{lemma}\label{l3}
Let $f_1(t)$, $f_2(t) \in  C^1([0,T])$ be $T-$periodic functions satisfying \eqref{0-hip1} and assume that $\big(\tilde{u}(t), \tilde{v}(t)\big)$ is a positive $T-$periodic globally attracting coexistence state of system \eqref{0.9} such that \eqref{0-hip2} holds. Then, for all $k \in \mathbb{N}^+$, there exists a value $\chi_c^{(k)}\in \mathbb{R}$ such that if $\chi_1 > \chi_c^{(k)}$, the first component $C_1^{(k)}$ of the solution to system \eqref{4.5} has an exponential growth rate over a set of times $I_k \subset [0,T]$. 
\end{lemma}

\begin{proof}\hfill\break
\hspace*{1em} For a given $k \in \mathbb{N}^+$, we consider the equivalent formulation of system \eqref{4.5} given in \eqref{4.7}. We seek to apply Lemma \ref{l1-aux}, using the properties of $\mathcal{Q}_k$ proved in Lemma \ref{l2}. In particular, we define the threshold
\begin{equation}\label{4.9}
\chi_c^{(k)} := \inf_{t \in [0,T] }  \chi_1^{(k)} (t),
\end{equation}
where we recall for all $t \in [0,T]$, $\chi_1^{(k)}(t)$ denotes the unique solution to the equation $\mathcal{Q}_k(\chi_1,t) = 0$. Since $\mathcal{Q}_k(\chi_1,t) = a_k(t)\chi_1 + b_k(t)$, with $a_k(t)<0$ over $[0,T]$, then for any $t \in [0,T]$
$$
\mathcal Q_k(\chi_1,t)<0
\quad \text{if and only if} \quad
\chi_1>\chi_1^{(k)}(t).
$$
As a result, for any $\chi_1 > \chi_c^{(k)}$, there exists at least one time $t_0 \in [0,T]$ such that $\chi_1 > \chi_1^{(k)}(t_0)$, and thus $\mathcal{Q}_k(\chi_1,t_0)<0$. Moreover, due to the continuity of $\mathcal{Q}_k(\chi_1, \cdot)$ on $[0,T]$ asserted by Lemma \ref{l2}, there exists at least one interval $I_k\subset[0,T]$ containing $t_0$ such that $\mathcal{Q}_k(\chi_1,t)<0$, for all $t \in I_k$. Hence, by Lemma \ref{l1-aux}, the first component $C_1^{(k)}$ exhibits an exponential growth rate on $I_k$.
\end{proof}
\begin{remark}\label{r1}
For a given $\chi_1 > \chi_c^{(k)}$, the set of instability times associated with the $k$-th mode,
$$
I_k = I_k(\chi_1)
:=
\{\, t \in [0,T] : \mathcal{Q}_k(\chi_1,t)<0 \,\},
$$
is, in general, not a single interval. In general, it depends on the value of $\chi_1$ and of the function $\chi_1^{(k)}(t)$.
\end{remark}
A counterpart to Lemma \ref{l3} with respect to the $C_2^{(k)}$ equation can be similarly proved in the same way. We omit the details of the proof, as it follows the same steps, based on second-order formulation \eqref{4.7} and the solvability of $\widehat{\mathcal{Q}}_k(\chi_1,t) = 0$ for all $t \in [0,T]$ granted by Lemma \ref{l2}.

\begin{lemma}\label{l4}
Suppose that $f_1(t)$, $f_2(t) \in  C^1([0,T])$ are $T-$periodic functions fulfilling \eqref{0-hip1} and that $\big(\tilde{u}(t), \tilde{v}(t)\big)$ is a positive $T-$periodic globally attracting coexistence state of system \eqref{0.9} satisfying assumption \eqref{0-hip2}. Then, for all $k \in \mathbb{N}^+$, there exists $\widehat{\chi}_c^{~(k)}\in \mathbb{R}$ such that if $\chi_1 > \widehat{\chi}_c^{~(k)}$, the second component $C_2^{(k)}$ of the solution to system \eqref{4.5} has an exponential growth rate over a set of times $\widehat{I}_k \subset [0,T]$. 
\end{lemma}

As in Lemma \ref{l3}, for each $k \in \mathbb{N}^+$, $\widehat{\chi}_c^{~(k)}$ is defined as
\begin{equation}\label{4.9-2}
\widehat{\chi}_c^{~(k)} := \inf_{t \in [0,T] } \widehat{ \chi}_1^{\,(k)} (t).
\end{equation}
We also emphasize that as described in Remark \ref{r1}, in general for a given $k \in \mathbb{N}^+$ and $\chi_1 > \widehat{\chi}_c^{~(k)}$ the set of instability times $\widehat{I}_k$ for $C_2^{(k)}$ need not be a single interval.
\\\\
Lastly, for the semitrivial state $\big(\tilde{u}(t),0\big)$, as a consequence of Lemma \ref{l2-2}, only a similar result to Lemma \ref{l3} can be proved, ruling out any instability arising from the growth of $C_2^{(k)}$. Again, we omit the details of the proof, as it follows the same structure.
\begin{lemma}\label{l5} Assume $f_1(t)$, $f_2(t) \in  C^1([0,T])$ are $T-$periodic functions satisfying \eqref{0-hip1} and that $\big(\tilde{u}(t), \tilde{v}(t)\big) = \big( \tilde{u}(t), 0\big)$ is a globally attracting $T-$periodic semitrivial state of system \eqref{0.9} verifying \eqref{0-hip2}. Then, for all $k \in \mathbb{N}^+$, there exists a value $\chi_c^{(k)}\in \mathbb{R}$ such that if $\chi_1 > \chi_c^{(k)}$, the first component $C_1^{(k)}$ of the solution to system \eqref{4.5} has an exponential growth rate over a set of times $I_k \subset [0,T]$. 
\end{lemma}

The proof of Theorem \ref{t1} now becomes an immediate consequence of Lemmas \ref{l3} and \ref{l4} in the case of positive coexistence, and of Lemma \ref{l5} for the semitrivial state.
\\\\
\textit{Proof of Theorem \ref{t1}}
\\
\hspace*{1em} We first consider the case in which $\big(\tilde{u}(t), \tilde{v}(t) \big)$ represents a coexistence state. Then, defining 
$\chi_{\text{crit}}$ as 
\begin{equation}\label{chi-def1}    
\chi_{\text{crit}} := \inf_{k \in \mathbb{N}^+} \left \{ \chi_c^{(k)}, \widehat{\chi}_c^{~(k)} \right\},
\end{equation}
where, for each $k \in \mathbb{N}^+$, $\chi_c^{(k)}$ and $ \widehat{\chi}_c^{~(k)}$ are as respectively defined in Lemma \eqref{4.9} and \eqref{4.9-2}. Then, for any $\chi_1 > \chi_{\text{crit}}$, we introduce
$$\mathcal{K}_1 := \left \{\, k \in \mathbb{N}^+, ~ \chi_1> \chi_c^{(k)} \,\right \} , \quad \mathcal{K}_2 := \left \{\, k \in \mathbb{N}^+, ~ \chi_1> \widehat{\chi}_c^{~(k)} \,\right \}.$$
In this way, $\mathcal{K} := \mathcal{K}_1 \cup \mathcal{K}_2$, which is clearly nonempty by definition of $\chi_{\text{crit}}$, is the set of unstable modes. As a result, by considering the instability time sets
\begin{equation}\label{4p-1.2}
\mathcal{I}_1 := \bigcup_{k\in \mathcal{K}_1} I_k, \quad \mathcal{I}_2 := \bigcup_{k \in\mathcal{K}_2} \widehat{I}_k,
\end{equation}
we may define 
\begin{equation}\label{4p-2}
    \mathcal{I} := \mathcal{I}_1 \cup \mathcal{I}_2.
\end{equation}
Consequently, by construction of $\mathcal{I}$, for all $t \in \mathcal{I}$, there exists at least one mode $k \in \mathcal{K}$ such that the $k-$th modal amplitude $\big(C_1^{(k)}(t), C_2^{(k)}(t)\big)$ has an exponential growth rate at time $t$ in at least one of its components. This implies the transient linear instability of the state $(\tilde{u}, \tilde{v})$ over $\mathcal{I}$. 
\\\\
For the case of a semitrivial state, if $\big(\tilde{u}(t),\tilde{v}(t)\big) = \big(\tilde{u}(t),0\big)$, then the threshold $\chi_{\text{crit}}$ can be similarly defined, taking into account that by Lemma \ref{l2-2}, $\widehat{\mathcal{Q}}_k(\chi_1,t) >0$ for all $k \in \mathbb{N}^+$, all $t>0$ and all $\chi_1 \in \mathbb{R}$. Thus, the only instability that can be characterized by Lemma \ref{l1-aux} comes from $\chi_c^{(k)}$, and consequently we define
\begin{equation}\label{chi-def2}    
\chi_{\text{crit}} := \inf_{k \in \mathbb{N}^+} \left \{ \chi_c^{(k)} \right\}.
\end{equation}
All the previous steps can be reproduced here, but only for the growth of the modes $C_1^{(k)}$, as in this case $\mathcal{K} = \mathcal{K}_1$ and $\mathcal{I} = \mathcal{I}_1$.
\qed
\begin{remark}\label{r2}
The significance of Theorem \ref{t1} lies in the fact that the instability mechanism is intrinsically linked to the temporal periodicity of the spatially homogeneous state. Indeed, if  $\chi_1 >\chi_\text{crit}$, then there exists a nonempty set of times $\mathcal{I} \subset [0,T]$ on which at least one spatial mode exhibits exponential growth. Since all coefficients in the linearized problem are $T-$periodic, the same instability mechanism reappears during every subsequent period. Consequently, the state $(\tilde{u}, \tilde{v})$ is linearly unstable over each translated set of times
$$
\mathcal{I}+kT:=\left\{t+kT ~:~ t\in \mathcal{I}\right\},\quad k\in \mathbb{N}_0
$$
In this sense, Theorem \ref{t1} establishes the existence of a recurrent sequence of transient instability time sets generated by the external periodicity of functions $f_1$ and $f_2$.
\end{remark}

\section{Numerical examples}\label{s5}

After establishing the proof of Theorem \ref{t1} in the previous section, we now present different numerical simulations aimed at identifying the threshold value $\chi_{\text{crit}}$ and its associated instability time set. This allows us to explore the periodic arising of different spatial patterns that are generated by the dynamics of system \eqref{0.8}. We consider both cases for the spatially homogeneous state $(\tilde{u}, \tilde{v})$, periodic coexistence, where both components are positive, and competitive exclusion, corresponding to a semitrivial state.

\subsection{Periodic coexistence}
We begin by considering the case in which the ODE system \eqref{0.9} admits a unique, positive, globally attracting, periodic coexistence state, $\big(\tilde{u}(t), \tilde{v}(t) \big)$. To ensure its existence and uniqueness, we assume that $a_i$ and $f_i$ are such that \eqref{1.5} holds. 
\\\\
We analyze two different examples. In the first, the functions $f_i$, representing the time variation of the carrying capacities, are in phase, attaining their maxima and minima simultaneously. In the second, we consider out-of-phase dynamics, so that the most favorable environmental conditions for one species coincide with the least favorable conditions for the other.
\subsubsection{Example I: In-phase carrying capacities}\label{s4e1}
To begin the first example with synchronized carrying capacities, we consider a large enough one-dimensional domain, given by $\Omega = (0,50)$, and the following parameters
\begin{equation}\label{5.1}
\begin{gathered}
a_1 = 0.25, \quad a_2 = 0.2, \quad \mu_1 = 1, \quad \mu_2 = 2.5,
\quad \alpha = 1, \quad \beta = 0.8,\\[1.5ex]
\chi_2 = 1.5, \quad 1+f_1(t) = 1-0.6\sin(0.25t),
\quad 1+f_2(t) = 1-0.5\sin(0.25t)\,.
\end{gathered}
\end{equation}
Parameters \eqref{5.1} represent for instance two different competing phenotypes of a bacterial species. Phenotype I, corresponding to the first species $u$ presents a slower reproduction rate $\mu_1$ than Phenotype II, which can proliferate faster. Both have relatively similar (weak) competition coefficients $a_1$, $a_2$, and self-production rates of the chemoattractant $\alpha$ and $\beta$. The environmental conditions, encoded by the time-dependent carrying capacities $1+f_i$ are synchronized, and affect both phenotypes in a similar way, with a period of $T = 8\pi$. Phenotype I is however slightly more responsive to environmental changes, attaining its maximum carrying capacity of $1+0.6$ at times $t = 6\pi + kT$, $k \in \mathbb{N}_0$ and minimum $0.4$ at $t = 2\pi + kT$, $k \in \mathbb{N}_0$, compared to phenotype II, whose carrying capacity oscillates between 0.5 and 1.5. Phenotype II presents a moderate chemotactic sensitivity, with a coefficient $\chi_2 = 1.5$, and we recall that the diffusion coefficient of both phenotypes is normalized to 1.
\\\\
Our aim is thus to determine the critical chemotactic sensitivity coefficient of the first phenotype that leads to the instability of the periodic coexistence state $\big(\tilde{u}(t), \tilde{v}(t)\big)$. To do so, we first verify that conditions \eqref{1.5} are satisfied by our choice of $a_i$ and $f_i$, in order to ensure that such a unique coexistence state exists and that it is globally attracting. In particular, it is direct to check that
$$
1+f_1^L = 0.4 >0.25 \cdot 1.5 =a_1\big(1+f_2^M\big) \quad \text{and} \quad 1+f_2^L = 0.5 > 0.2 \cdot 1.6 = a_2\big(1+f_1^M\big).
$$
Moreover, we numerically verify that hypothesis \eqref{0-hip2} holds, to grant the solvability of equations \eqref{4-eq}, leading to the thresholds $\chi_c^{(k)}$ and $\widehat{\chi}_1^{~(k)}$ for each mode. To approximate the state $\big(\tilde{u}(t), \tilde{v}(t)\big)$, which is the unique global attractor of system \eqref{0.9}, we choose arbitrary initial data and numerically integrate the system over successive periods. We then evaluate the solution at times $t_k =kT$, $k \in \mathbb{N}^+$, until the difference between two consecutive evaluations falls below a prescribed tolerance. This provides us an approximation of initial value $(u_0,v_0)$ that generates the periodic coexistence state $\big(\tilde{u}(t), \tilde{v}(t)\big)$. 
\\\\
Computing $\big(\tilde{u}(t), \tilde{v}(t)\big)$ in this way, we verify that \eqref{0-hip2} holds, as depicted in Figure \ref{fig:f1}. In particular, we find that the red curve, corresponding to $\alpha \big(1+f_2(t)\big)$ lies below the blue curve, $\alpha a_2 \tilde{u}(t) + \big(2\alpha - a_2 \beta\big) \tilde{v}(t)$, for all times.\\
\begin{figure}[htp]
    \centering
    \includegraphics[width=10 cm]{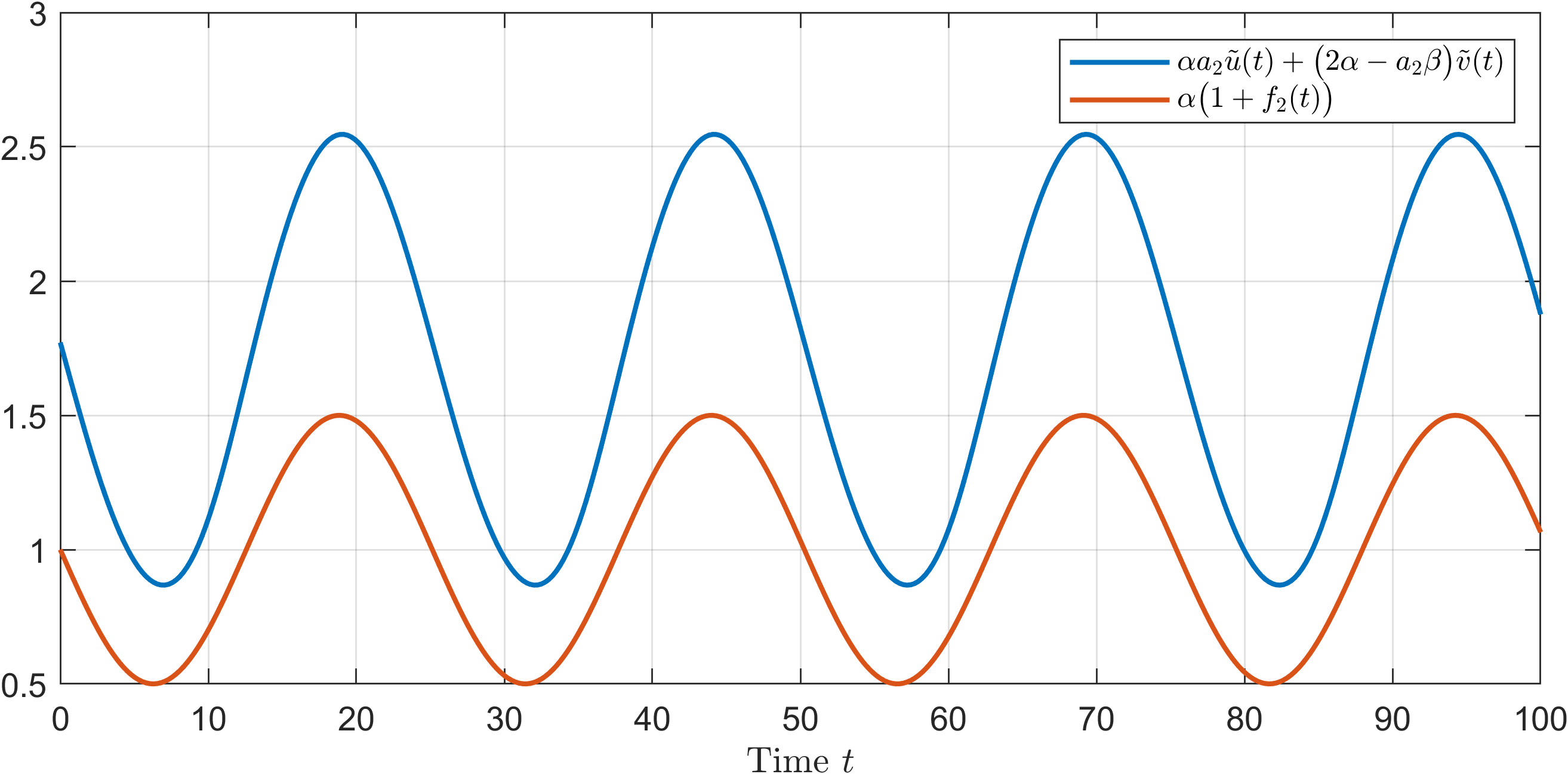}
    \caption{Numerical verification of hypothesis \eqref{0-hip2} for $(\tilde{u},\tilde{v})$, the unique coexistence state associated to the parameter set \eqref{5.1}.}
    \label{fig:f1}
\end{figure}\\
The verification of hypothesis \eqref{0-hip2} allows us to compute the values of $\chi_1^{(k)}(t)$ and $\widehat{\chi}_1^{~(k)}(t)$, as defined in Lemma \ref{l2}. Specifically, we discretize the interval $[0,T] = [0,8\pi]$ by considering 500 evenly spaced nodes, over which we solve both linear equations \eqref{4-eq}. This is done for all modes $k \in \{1, \dots, 30\}$, and the resulting functions are plotted in Figure \ref{fig:f2}. The values of $\chi_1^{(k)}(t)$ are shown in the top row, while those of $\widehat{\chi}_1^{~(k)}(t)$ are displayed in the bottom row. For clarity, only a selected subset of modes is represented, namely $k \in \{1,\dots,5\}$ on the left panel, $k \in \{6, \dots, 12\}$ in the center, and $k \in \{18, \dots, 22\}$ on the right.\\
\begin{figure}[htp]
    \centering
    \includegraphics[width=15 cm]{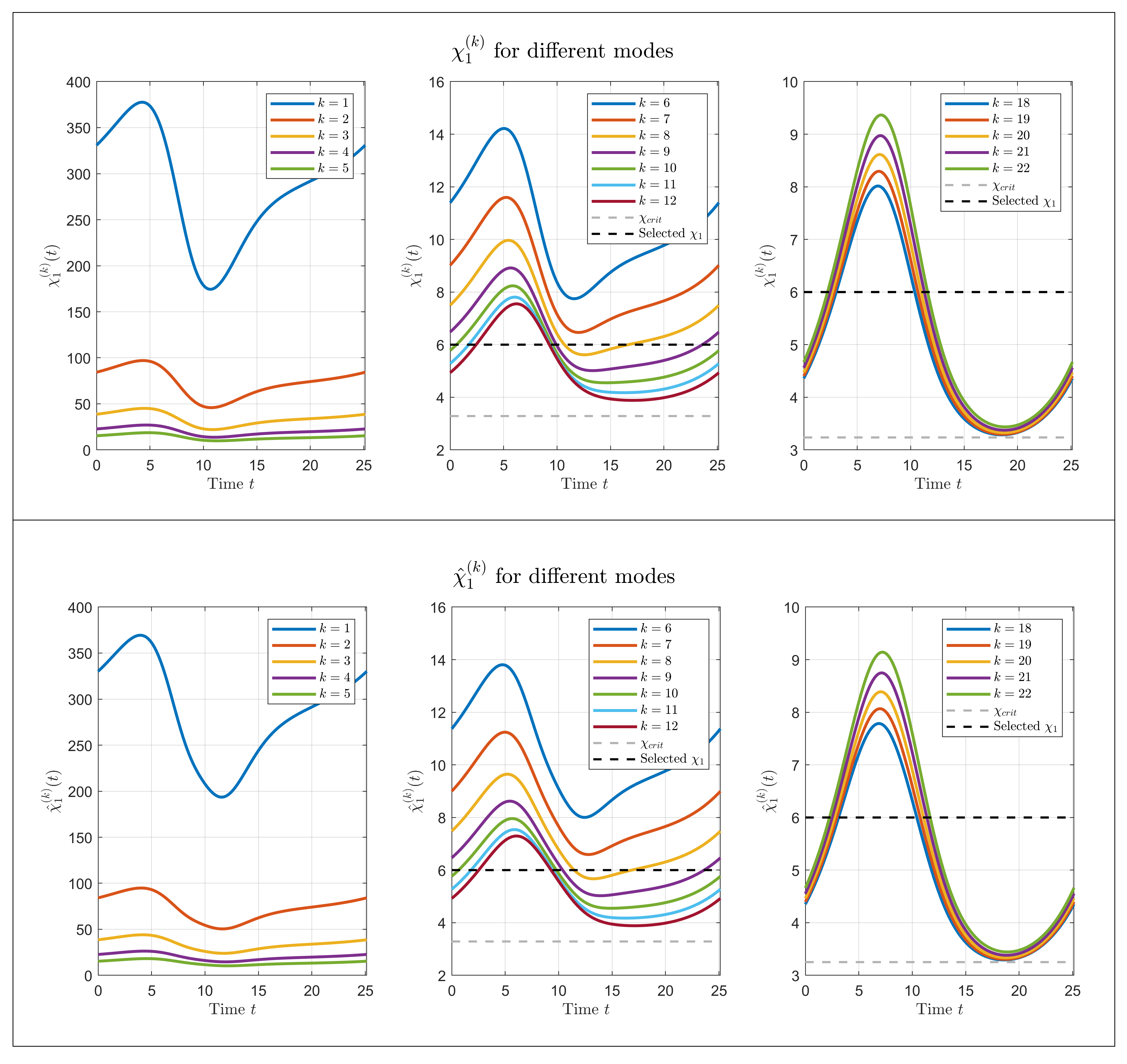}
    \caption{Values of $\chi_1^{(k)}(t)$ and $\widehat{\chi}_1^{~(k)}(t)$ for different modes $k$ for the parameter set \eqref{5.1}. The values of $\chi_\text{crit}\approx 3.287$, the instability threshold, and $\chi_1 = 6$, the parameter selected for the numerical simulations, are also represented for reference.}
    \label{fig:f2}
\end{figure}\\
For all modes $k$ considered, the functions $\chi_1^{(k)}(t)$ and $\widehat{\chi}_1^{~(k)}(t)$ exhibit very similar behavior. The leading modes require very large values of $\chi_1$ to be destabilized. In particular, both $\chi_1^{(1)}(t)$ and $\widehat{\chi}_1^{~(1)}(t)$ lie above $200$ for the majority of $[0,T]$, while the corresponding values for $k = 2$ they are of order $100$. As $k$ increases, the magnitude of $\chi_1^{(k)}(t)$ and $\widehat{\chi}_1^{~(k)}(t)$, decreases rapidly at first, and then begins to stabilize for intermediate vales of $k$. In particular, as shown on the central panels, for $k \in \{8, \dots, 12\}$, the graphs of $\chi_1^{(k)}(t)$ and $\widehat{\chi}_1^{~(k)}(t)$ become progressively closer. This is visible again on the right-hand panels, for $k \in \{18,\dots, 22\}$, although with the opposite behavior, as for those modes, $\chi_1^{(k)}(t)$ and $\widehat{\chi}_1^{~(k)}(t)$ begin to increase with $k$.
\\\\
To better illustrate this growth in the magnitude of $\chi_1^{(k)}(t)$ and $\widehat{\chi}_1^{~(k)}(t)$ with respect to $k$, Figure \ref{fig:f2-2} depicts the quantities
$$
\chi_c^{(k)}:= \inf_{t\in[0,T]} \chi_1^{(k)}(t), \quad \widehat{\chi}_c^{\,(k)} :=\inf_{t\in[0,T]} \widehat{\chi}_1^{~(k)}(t),
$$
as defined in \eqref{4.9} and \eqref{4.9-2}. For clarity, only the modes $k \in \{8, \dots, 30\}$ are represented.\\
\begin{figure}[htp]
    \centering
    \includegraphics[width=11 cm]{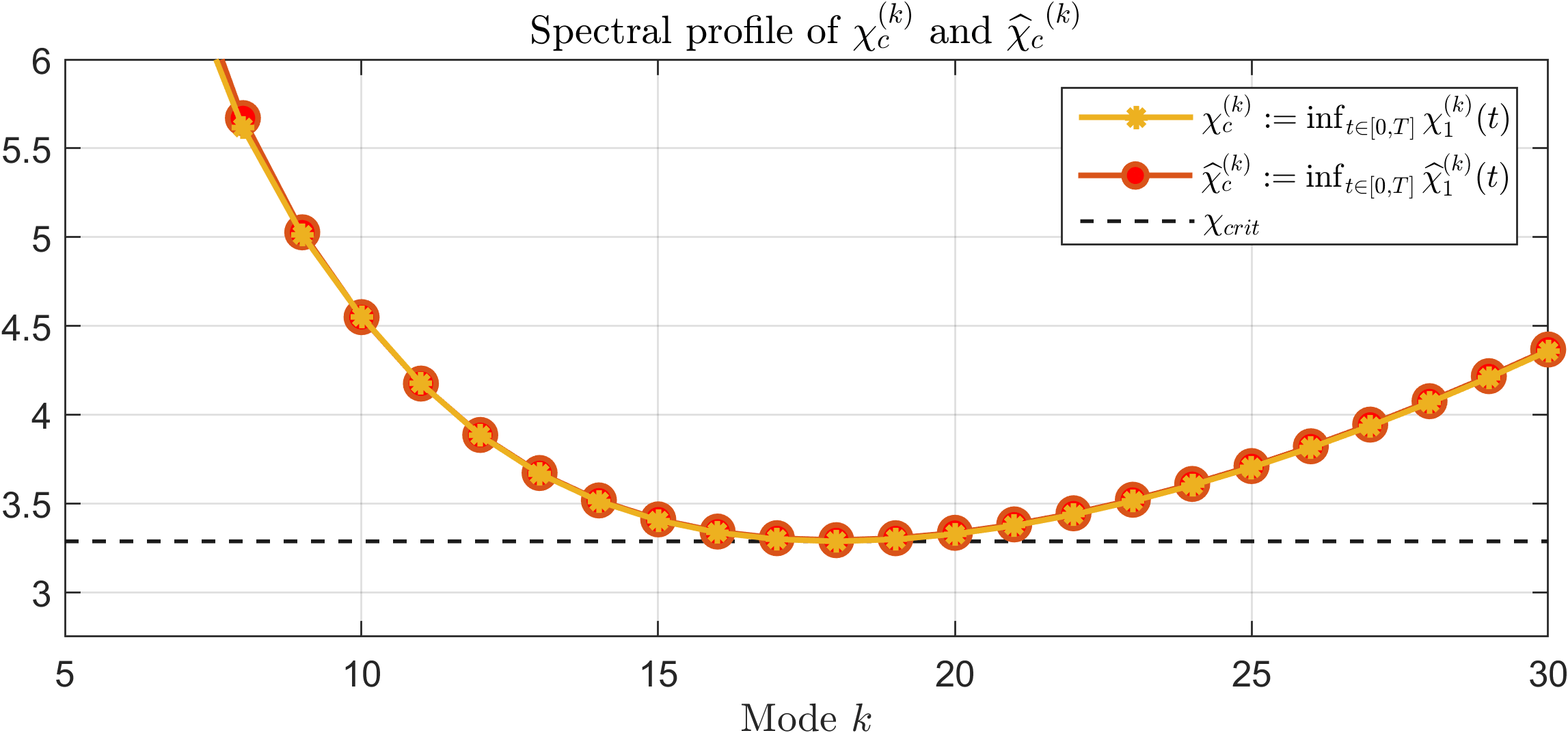}
    \caption{Modal distribution of $\chi_c^{(k)}$ and $\widehat{\chi}_c^{\,(k)}$ for $k \in \{8, \dots, 30\}$, computed for the parameter set \eqref{5.1}.}
    \label{fig:f2-2}
\end{figure}
\\The similarity of the curves $\chi_1^{(k)}(t)$ and $ \widehat{\chi}_1^{~(k)}(t)$ is clearly visible here, as the yellow asterisks, indicating the values of $\chi_c^{(k)}$ lie very close to the red circles, corresponding to $\widehat{\chi}_c^{\,(k)}$. In particular, the threshold $\chi_{\text{crit}}$, defined in \eqref{chi-def1}, can be identified by the minimum value of the results shown in Figure \ref{fig:f2-2}. This reveals that  
$$
\chi_\text{crit}  = \inf_{k \in \mathbb{N}^+} \left \{ \chi_c^{(k)}, \widehat{\chi}_c^{~(k)} \right\} =\widehat{\chi}_c^{~(18)}\approx 3.287,
$$
Hence, the first modal amplitude to become linearly unstable is $C_2^{(18)}$, corresponding to the eighteenth spatial mode of $v$.
\\\\
Referring back to Figure \ref{fig:f2}, we represent the value of $\chi_\text{crit}$ in gray dashed lines on the center and right panels. As it can be seen, the minimum is precisely attained in the bottom right panel, for the mode $k=18$, shortly before $t = 20$. As $\chi_c^{(18)}$ is also very close to the infimum, a similar behavior is found on the top right panel for $\chi_1^{(18)}(t)$. Consequently, Theorem \ref{t1} implies that, for any $\chi_1 > \chi_{\text{crit}}$, there exists an instability time set $\mathcal{I} \subset [0,T]$ and a collection of modes $\mathcal{K}$ over which the periodic coexistence state $(\tilde{u},\tilde{v})$ becomes linearly unstable. 
\\\\
In the previous phenotypical interpretation of the model, the value $\chi_{\text{crit}} \approx 3.287$ represents the minimum chemotactic sensitivity required for Phenotype I to destabilize the homogeneous coexistence state $(\tilde{u}, \tilde{v})$. Since this threshold is more than twice the value of the chemotactic sensitivity of Phenotype II ($\chi_2 = 1.5$), instability occurs only when Phenotype I exhibits a substantially stronger tendency to move along chemoattractant gradients. The instability mechanism can therefore be interpreted as arising from the interaction between a highly chemotactic-responsive but moderately reproducing phenotype, and a less motile phenotype with a higher reproduction rate. Consequently, if the chemotaxis sensitivity of Phenotype I is above $\chi_{\text{crit}}$, Theorem \ref{t1} predicts the emergence of transient bacterial aggregations during specific portions of each environmental cycle, which reappear on each subsequent cycle due to the environmental periodicity.
\\\\
With the aim of verifying this behavior numerically by computing the solution to system \eqref{0.8}, we take for instance $\chi_1 = 6 > \chi_{\text{crit}}$. For reference, this value is indicated by the black dashed line in the central and right panels of Figure \ref{fig:f2}. The corresponding instability set $\mathcal{I} \subset [0,T]$ can then be identified directly from the graphs of $\chi_1^{(k)}$ and $\widehat{\chi}_1^{~(k)}$, consisting of all times for which at least one of these curves lies below the level $\chi_1 = 6$. In particular, we numerically obtain that
\begin{equation}\label{5-i1}
    \mathcal{I} = [0,t_1) \cup (t_2,T] \approx [0,3.412) \cup ( 9.219,25.132].
\end{equation}
Therefore, Theorem \ref{t1} predicts the arise of a transient instability during the initial and final portions of each environmental cycle. Moreover, as stated in Remark \ref{r2}, the instability set is translated over each environmental cycle, and since the interval $(t_2,T]$ from the first period is followed by the interval $[T, t_1+T)$, we find that for the complete evolution problem over $[0,\infty)$, the coexistence state $(\tilde{u},\tilde{v})$ is linearly unstable over the set
\begin{equation}\label{5.2}
\mathcal{I}_\infty := [0, t_1) \cup (t_2,t_1+T) \cup (t_2+T, t_1+2T) \cup (t_2 + 2T, t_1+3T) \cup \dots    
\end{equation}
We remark again that, since Lemma \ref{l1-aux} only provides a sufficient condition for the growth of the modal amplitudes, stability is not ensured over $[0,\infty) \setminus \mathcal{I}_\infty$. In addition to $\mathcal{I}$, we also approximate the sets $\mathcal{I}_1$ and $\mathcal{I}_2$ defined in \eqref{4p-1.2}, which collect the times at which at least one modal amplitude $C_1^{(k)}$ or $C_2^{(k)}$, respectively, has an exponential growth rate. Therefore $\mathcal{I}_1$ and $\mathcal{I}_2$ respectively characterize the sets of transient instability of $u$ and $v$. A numerical approximation yields
\begin{equation}\label{5-i2}
    \mathcal{I}_1 = [0, 3.283) \cup (9.219,25.132], \quad \mathcal{I}_2 = [0, 3.412) \cup (9.219, 25.132].
\end{equation}
As we see, given the high similarity of $\chi_1^{(k)}(t)$ and $\widehat{\chi}_1^{~(k)}(t)$, both instability sets $\mathcal{I}_1$ and $\mathcal{I}_2$ are very similar, and in fact their second half coincides. However, $u$ may loose instability at $t = 3.283$, just shortly before $v$, at $t = 3.412$. With the aim of extending the analysis over more than one environmental cycle, in the same way that $\mathcal{I}_\infty$ was defined in \eqref{5.2}, we consider the corresponding analogous $\mathcal{I}_\infty^{(1)}$ and $\mathcal{I}_\infty^{(2)}$ for $\mathcal{I}_1$ and $\mathcal{I}_2$, respectively.
\\\\
To investigate the relationship between the environmental conditions and the transient source of instability, Figure \ref{fig:f3} shows the carrying capacities $1+f_1$ and $1+f_2$ and the periodic coexistence states $(\tilde{u},\tilde{v})$, aside the instability sets $\mathcal{I}_\infty^{(1)}$ and $\mathcal{I}_\infty^{(2)}$, represented in red over the time axis. The complement sets $[0,\infty) \setminus \mathcal{I}_\infty^{(1)}$ and $[0,\infty) \setminus \mathcal{I}_\infty^{(2)}$ ---the only time sets in which linear stability may hold--- are represented in green. Here and below, we refer to the complement sets $[0,\infty) \setminus \mathcal{I}_\infty^{(i)}$ as non-instability sets. This should not be interpreted as implying stability on $[0,\infty)\setminus\mathcal{I}_\infty^{(i)}$, rather, it only indicates that the instability criterion provided by Lemma \ref{l1-aux} does not detect exponential growth over them. Nevertheless, any intervals on which linear stability may occur must necessarily be contained in these non-instability sets.\\
\begin{figure}[htp]
    \centering
    \includegraphics[width=14 cm]{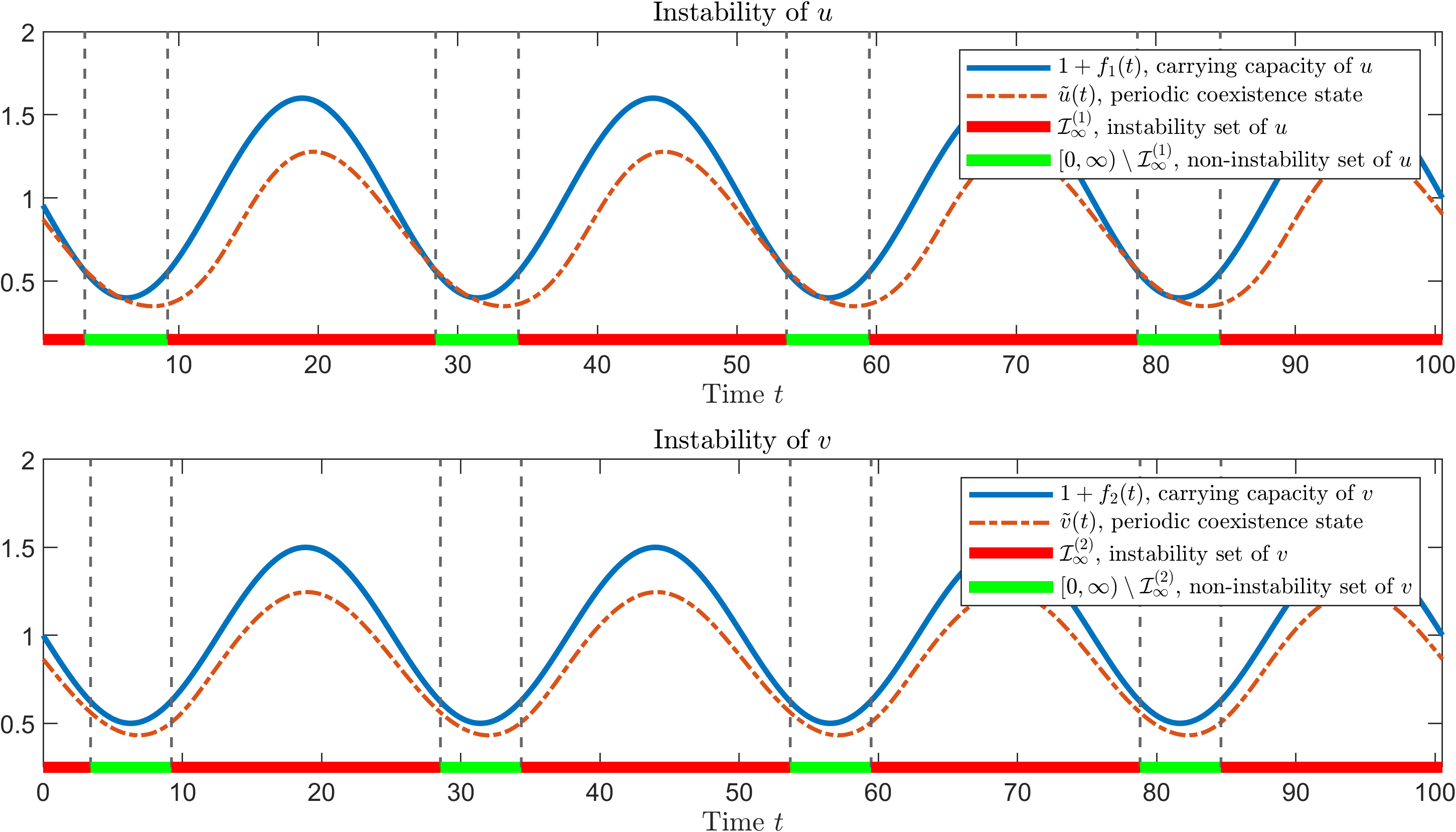}
    \caption{Effect of environmental conditions on the transient instability of $u$ and $v$, computed with parameters \eqref{5.1} and $\chi_1 = 6>\chi_{\text{crit}}$. Top panel: in blue $1+f_1(t)$, the carrying capacity of the first species; in red dashed line $\tilde{u}(t)$, the first component of the periodic coexistence state; on the time axis, in red $\mathcal{I}_\infty^{(1)}$, the extended instability set of $u$, and in green its complement $[0,\infty) \setminus \mathcal{I}_\infty^{(1)}$. Bottom panel: analogous for $v$, with instability set $\mathcal{I}_\infty^{(2)}$.}
    \label{fig:f3}
\end{figure}\\
The results seem to indicate how the instability sets $\mathcal{I}_\infty^{(1)}$ and $\mathcal{I}_\infty^{(2)}$ occur in phases where the respective carrying capacities $1+f_1$ and $1+f_2$ attain sufficiently large values, while on the contrary, the only possible stability sets, $[0,\infty) \setminus \mathcal{I}_\infty^{(1)}$ and $[0,\infty) \setminus \mathcal{I}_\infty^{(2)}$ are observed only near their minima. This could suggest that instability is only detected when the environmental conditions are sufficiently favorable for growth, and consequently the populations are not strongly constrained by low carrying capacities. However, as we shall see in Example II and in more detail in Section \ref{s6}, this is not generally true, and there is a particular balance that governs the location of the instability sets, rather than directly coinciding with high-enough carrying capacities. 
\\\\
After the analytical characterization of the transient stability provided by the sets $\mathcal{I}$, $\mathcal{I}_1$ and $\mathcal{I}_2$, we next turn to the simulation of the complete system \eqref{0.8} in order to observe wether such instability can give rise to recurrent spatial patterns. In particular, a second-order finite-difference scheme is used, with constant grid spacing $\Delta x = 0.05$ and time step $\Delta t = 2 \cdot 10^{-3}$. As initial data, we take a small perturbation of $\big(\tilde{u}(0), \tilde{v}(0)\big)$, given for instance by
\begin{equation}\label{5.3}
    u(x,0) = \tilde{u}(0) + \varepsilon \sin(\pi x/50), \quad v(x,0) = \tilde{v}(0) + \varepsilon \cos(\pi x/50),
\end{equation}
for $\varepsilon = 0.05$. This allows us to compare the obtained solution with the periodic coexistence state $(\tilde{u}, \tilde{v})$, which is globally attracting for the non-diffusive and non-chemotactic system \eqref{0.9}. The results are depicted in Figure \ref{fig:f4}, with the left panel showing the comparison $u - \tilde{u}$, the center panel portraying $v - \tilde{v}$, and the right panel $w - \tilde{w}$, where we recall that $\tilde{w}(t) := \alpha \tilde{u}(t) + \beta \tilde{v}(t)$. \\
\begin{figure}[htp]
    \centering
    \includegraphics[width=17 cm]{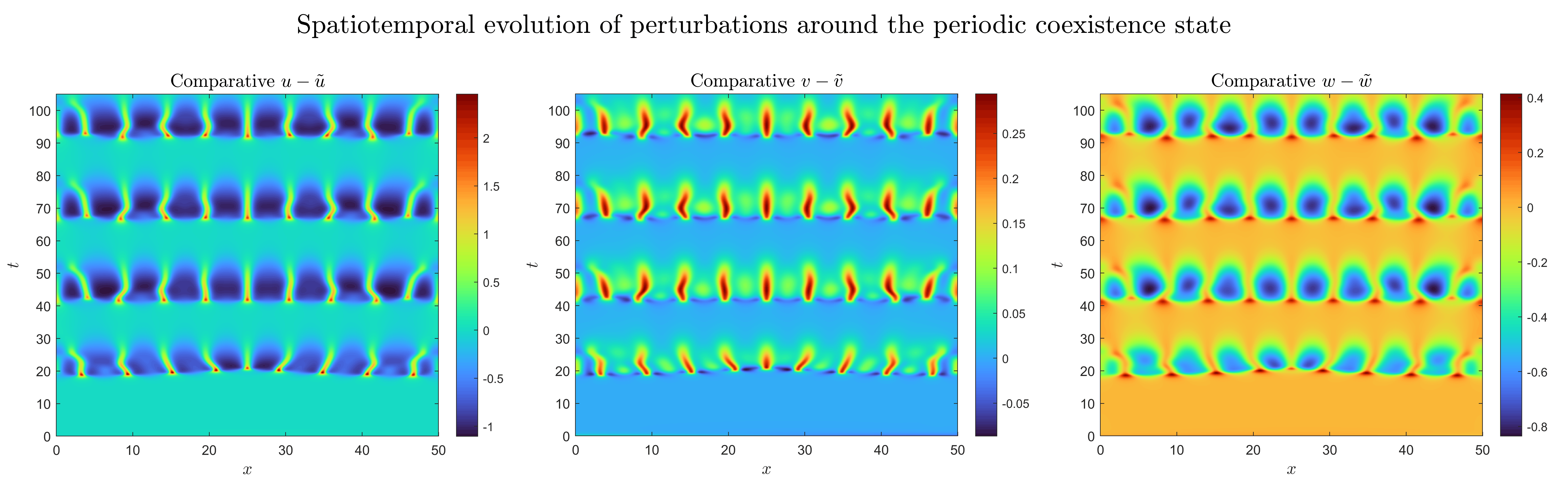}
    \caption{Comparison $u - \tilde{u}$, $v - \tilde{v}$, and $w - \tilde{w}$, where $(u,v,w)$ is the solution to system \eqref{0.8} with parameters \eqref{5.1}, $\chi_1 = 6 > \chi_{\text{crit}}$ and initial data \eqref{5.3}; $(\tilde{u}, \tilde{v})$ is the unique periodic coexistence state of system \eqref{0.9}; and $\tilde{w}(t) := \alpha \tilde{u}(t) + \beta \tilde{v}(t)$.}
    \label{fig:f4}
\end{figure}\\
As it can be observed on the three panels, after an initial period during which the solution remains close to the spatially homogeneous state, lasting until shortly before $t = 20$, the three components $u$, $v$ and $w$ develop highly structured spatial patterns. For $u$ and $v$, these take the form of nine symmetric population aggregates (plus two, one on each side of the boundary) that rise above the equilibrium levels $\tilde{u}$ and $\tilde{v}$, concentrating the majority of the population. In the case of $u$, the larger chemotaxis sensitivity ($\chi_1 = 6 > 1.5 = \chi_2$) leads to significantly stronger aggregations than those observed for $v$, with the maximum deviation from $\tilde{u}$ exceeding $2$, while the maximum difference between $v$ and $\tilde{v}$ is one order below, at $0.25$. With respect to $w$, the highest concentrations of the chemoattractant occur precisely in the regions where the population clusters are located, reflecting the combined effects of chemotactic attraction and self-production.
\\\\
This instability period, spanning approximately from $t = 20$ to $t = 30$, is next followed by a second stable regime, where spatial heterogeneities decay and the solution returns close to the homogeneous periodic state $(\tilde{u}, \tilde{v}, \tilde{w})$. At approximately $t=40$, spatial heterogeneities begin to reappear, giving rise to a new phase of pattern formation, somewhat longer than the first one. This alternation between stable and unstable periods, with recurrent intervals of strong aggregation, is observed until the end of the simulation, after four environmental cycles. Only the first unstable period is observed to be shorter than the other ones, of approximately equal length.
\\\\
To better analyze the effect of the instability sets $\mathcal{I}_\infty^{(1)}$ and $\mathcal{I}_\infty^{(2)}$ and their complements in the emergence of spatial heterogeneities, the left panel of Figure \ref{fig:f5} represents the temporal evolution of $\max_{x \in \Omega} u(x,t)$ and $\max_{x \in \Omega} v(x,t)$, aside the periodic coexistence state $(\tilde{u},\tilde{v})$, in dotted yellow line. The instability sets $\mathcal{I}_\infty^{(1)}$ and $\mathcal{I}_\infty^{(2)}$ are also represented on the time axis in red, as well as their complements, in green. \\
\begin{figure}[htp]
    \centering
    \includegraphics[width=16 cm]{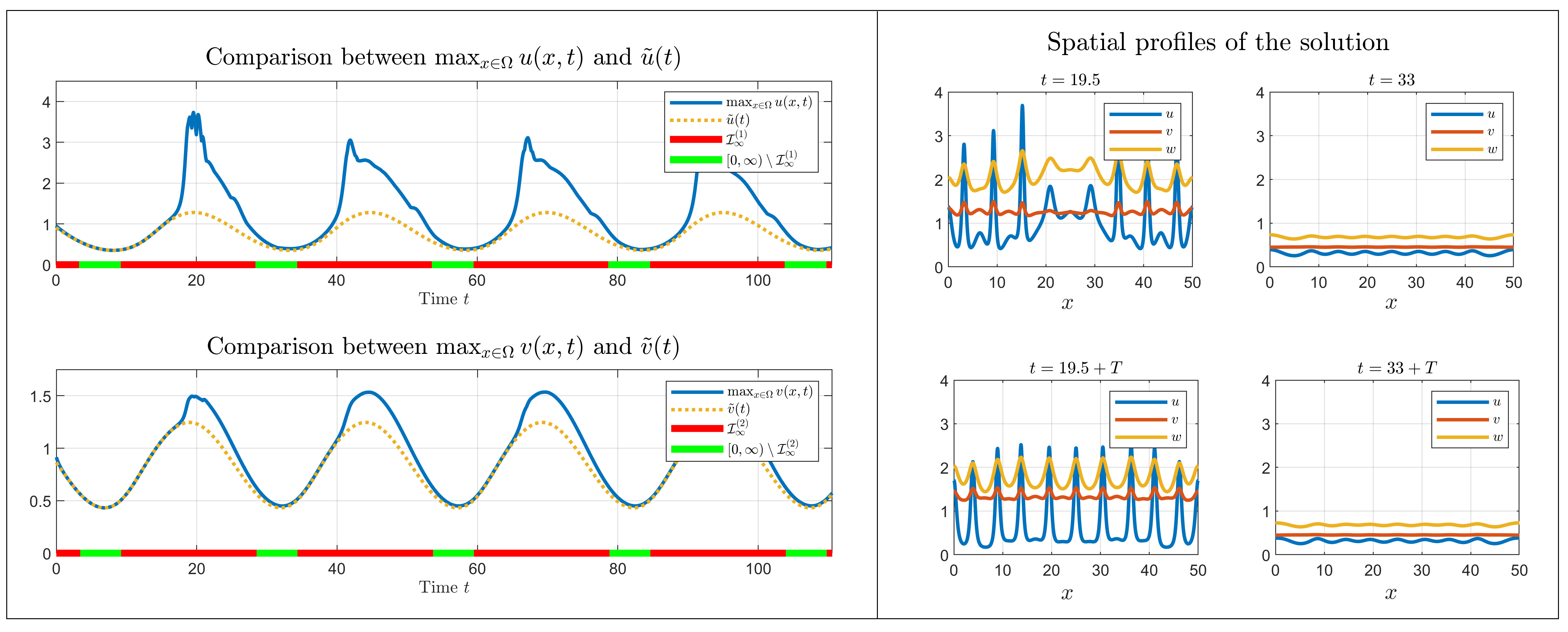}
    \caption{Left panel: evolution of $\max u(\cdot,t)$ and  $\max v(\cdot,t)$ aside the periodic coexistence state $(\tilde{u}, \tilde{v})$. Right panel: spatial profiles of the solution $(u,v,w)$ at $t = 19.5$, $t = 33$ and in the next environmental cycle. Computed with parameters \eqref{5.1}, initial data \eqref{5.3}, and $\chi_1 = 6$.}
    \label{fig:f5}
\end{figure}\\
During the initial instants of the simulation, although both components lie within their respective instability sets, spatial heterogeneity is minimal, with both maximum plots presenting no deviations from the state $(\tilde{u}, \tilde{v})$. This is followed by a first non-instability period, represented in green, which extends approximately until $t = 10$ (both intervals ending precisely at $9.219$, as computed in \eqref{5-i2}), where both solutions still remain close to the spatially homogeneous state. It is during the second instability period, when spatial heterogeneity starts to develop. Due to the strong taxis effect, as previously discussed, the destabilizing effect is greater for $u$, resulting in patterns that reach a maximum population density above $3.5$, approximately at $t = 19.5$. In contrast, the deviations of $v$ remain substantially scarcer, with its maximum density reaching only about $1.5$.
\\\\
To further analyze the shape of the emerging patterns, the spatial profiles of $u$, $v$ and $w$ are represented on the right panel of Figure \ref{fig:f5}. The first plot corresponds precisely to $t = 19.5$, the time at which the largest deviations of $u$ from $\tilde{u}$ were observed. As shown by the blue curve, several symmetric spike-like aggregates emerge, with the largest population densities nearly reaching $4$, close to $x = 15$ and $x = 35$. The corresponding profiles of $v$ and $w$, respectively represented in red and yellow, show substantially less spatial variation, especially in the case of $v$. The self-production of the chemoattractant by the two species results in similar high concentrations of $w$ located at the maxima of $u$, although significantly damped due to the comparatively homogeneous distribution of $v$. 
\\\\
After some time, as the periods of instability reach an end, the previously formed aggregates begin to dissipate and the maximum values gradually decrease, returning towards the spatially homogeneous profile $(\tilde{u}(t), \tilde{v}(t))$, as illustrated again in the left panel of Figure \ref{fig:f5}. During this subsequent stable regime, as already inferred from the plots in Figure \ref{fig:f4}, spatial heterogeneities decay. 
This is further confirmed by the spatial profiles at $t = 33$ (second plot on the right panel), where only small oscillations remain, negligible compared to the structures observed at $t = 19.5$.
\\\\
This alternation between growth and decay persists over the subsequent environmental cycles: instability phases generate pattern formation within the sets $\mathcal{I}_\infty^{(1)}$ and $\mathcal{I}_\infty^{(2)}$, while stability phases, corresponding to their complements, drive the solution back toward $(\tilde{u}, \tilde{v})$.
\\\\
One important dynamical feature is observed during the first environmental cycle. As previously noted, the first pattern formation process is slightly delayed with respect to the others, not reaching an aggregation stage until nearly $t = 20$, resulting in a longer initial stable stage, as opposed to the subsequent ones. This is most likely due to the fact that the first instability intervals are relatively short, ending at $t \approx 3.283$ for $u$ and $t \approx 3.412$ for $v$, and are followed by a prolonged non-instability phase, spanning until $t \approx 9.219$ in both cases. Combined with small initial perturbations, this prevents the early amplification of spatial heterogeneities, which only begin to develop near $t = 20$. As a result, the patterns that emerge on the first environmental cycle differ from those observed in later cycles.
\\\\
This effect is visible in both panels on Figure \ref{fig:f5}. On the left-hand side, the first significant deviation from $(\tilde{u}, \tilde{v})$ occurs approximately after $t = 18$ and its shape is not reproduced in subsequent cycles, which instead display more regular characteristics. This is also noticeable in the spatial profiles, on the right panel. In particular, the profiles at $t = 19.5 + T \approx 44.633$ are more evenly distributed than those at $t = 19.5$, and exhibit a reduced maximal amplitude, remaining below $3$, with the peak occurring earlier at approximately $t \approx 41.82$, as visible on the left panel.
\\\\
To investigate the abrupt changes observed in the graph of $\max u(\cdot,t)$ in Figure \ref{fig:f5} during the first growth phase,
we represent additional spatial profiles of $u$ near $t=20$ in the top row of Figure \ref{fig:f6}. Rapid changes in the bacterial density are observed, especially in the central region of the domain. In particular, after the profile at $t=19.5$, previously shown in Figure \ref{fig:f5}, the two central aggregates merge into a single cluster by $t=20$. At this stage, the maximum density has fallen, and is now close to $3$. Shortly afterwards, at $t=20.2$, the two neighboring aggregates experience a rapid increase in density, before partially dissipating again by $t=20.75$, due to the increase in the central cluster. These local reorganizations explain the non-smooth behavior observed in the evolution of $\max u(\cdot,t)$ during this interval.\\
\begin{figure}[htp]
    \centering
    \includegraphics[width=15 cm]{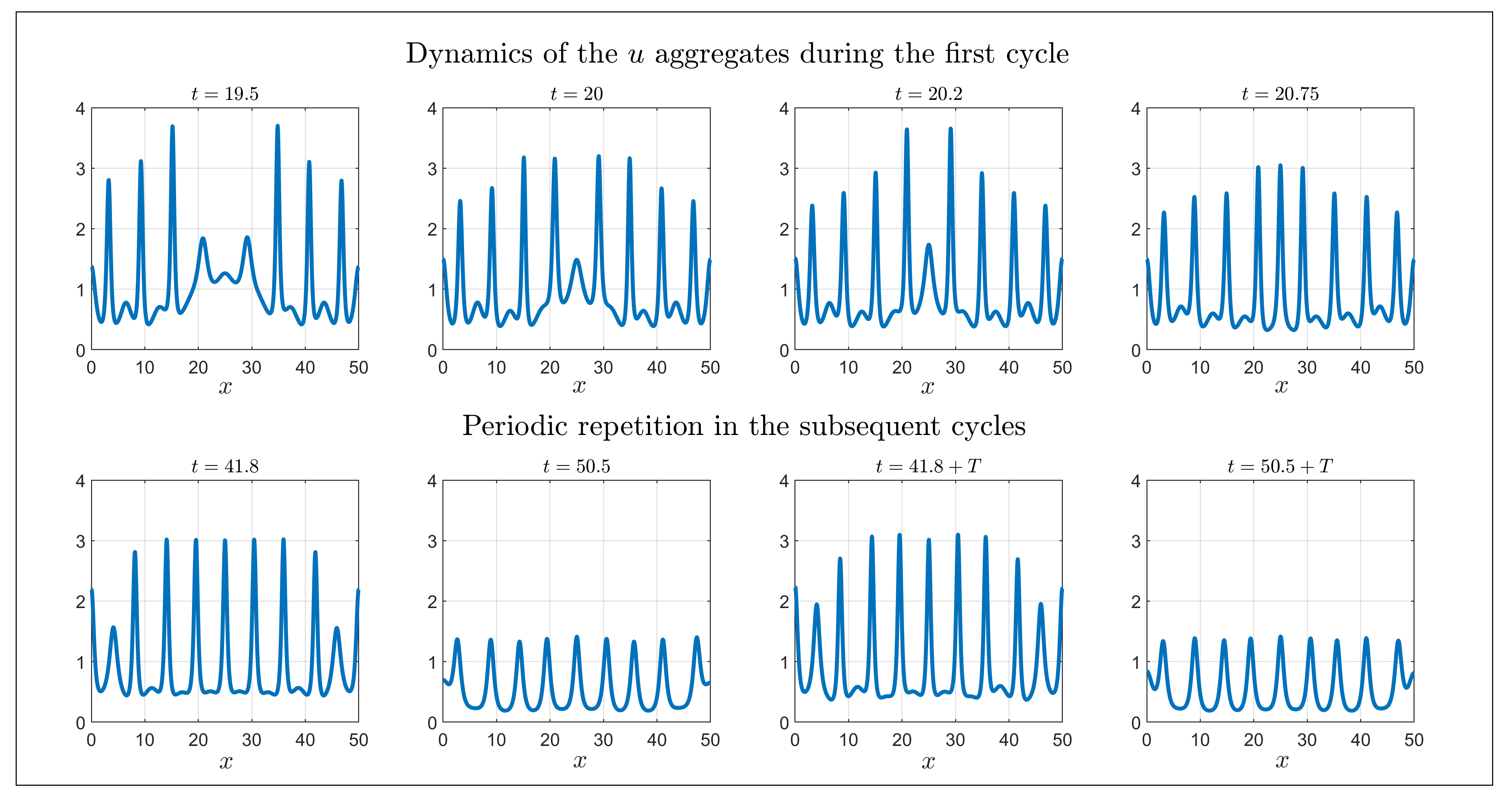}
    \caption{Top row: different spatial profiles of the $u$ component of the solution near $t = 20$. Bottom row: further spatial profiles one cycle apart.Computed with parameters \eqref{5.1}, initial data \eqref{5.3}, and $\chi_1 = 6$.}
    \label{fig:f6}
\end{figure}\\
To further analyze the recurrent dynamics observed in later environmental cycles, the bottom row of Figure \ref{fig:f6} shows the spatial profiles of $u$ at $t = 41.8$, corresponding to the largest value attained by $\max u(\cdot,t)$ during the second cycle, and at $t=50.5$, during the subsequent decay phase, together with the corresponding profiles one period later, namely at $t=41.8+T \approx 66.933$ and $t=50.5+T \approx 75.632$. The profiles obtained in both cycles are remarkably similar, in contrast to the behavior observed during the first cycle, shown in Figure \ref{fig:f5}.
\\\\
This suggests that, in this coexistence setting, after the initial transient dynamics, the pattern formation process may become periodic in time, synchronized with the environmental periodicity. Successive instability intervals give rise to highly structured aggregations, which are followed by stable phases during which spatial heterogeneity is substantially reduced and the dynamics of the system are driven by those of the non-diffusive and non-chemotactic ODE system \eqref{0.9}.
\\\\
We next turn to study the behavior of the system under non-synchronous carrying capacities, to investigate the changes produced in the instability sets.

\subsubsection{Example II: Out-of-phase carrying capacities}\label{s4e2}

To study the dynamics of system \eqref{0.8} under non-synchronous carrying capacities, we next consider the following parameters
\begin{equation}\label{5.4}
\begin{gathered}
 a_1 = 0.3, \quad a_2 = 0.15, \quad  \mu_1 = 1, \quad \mu_2 = 2.5
\quad \alpha = 1, \quad \beta = 1,\\[1.5ex]
\chi_2 = 2, \quad 1+f_1(t) = 1-0.3\sin(0.25t),
\quad 1+f_2(t) = 1+0.3\sin(0.25t)\,,
\end{gathered}
\end{equation}
where the carrying capacities $1+f_1(t)$ and $1+f_2(t)$ are out of phase, and the maxima of one coincide with the minima of the other. In the previous biological interpretation, one can think of two distinct phenotypes, adapted to opposite seasonal conditions. The purpose of this example is to analyze the location of each instability set and, as previously anticipated ---and contrary to what Figure \ref{fig:f3} could suggest--- conclude that these need not coincide with the maxima of each of the carrying capacities. 
\\\\
As in Example I, it is direct to check that the parameter set \eqref{5.4} satisfies conditions \eqref{1.5}, ensuring the existence of a unique, globally attracting, coexistence solution $(\tilde{u}, \tilde{v})$ to system \eqref{0.9}. Moreover, through a numerical integration, one can check that for such state $(\tilde{u}, \tilde{v})$, hypothesis \eqref{0-hip2} does also hold, granting the unique solvability of equations \eqref{4-eq} for all $t \in [0,T]$. By computing the solutions $\chi_1^{(k)}(t)$ and $\widehat{\chi}_1^{\, (k)}(t)$, we obtain that in this case
$$
\chi_{\text{crit}} =\widehat{\chi}_c^{\, (17)} \approx 3.2116
$$
This time, we select for instance $\chi_1 = 5 > \chi_{\text{crit}}$, which yields the following instability sets
\begin{equation}\label{5.5}
    \mathcal{I}_1 = [0, 3.536) \cup (10.229, 8\pi], \quad \mathcal{I}_2 = [0, 3.915) \cup (10.861,8\pi],
\end{equation}
with very similar values. It therefore follows that
\begin{equation}\label{5.6}
    \mathcal{I} = \mathcal{I}_1 \cup \mathcal{I}_2 = [0,  3.915)\cup (10.229, 8\pi].
\end{equation}
The extended instability sets $\mathcal{I}_\infty^{(1)}$ and $\mathcal{I}_\infty^{(2)}$, obtained through \eqref{5.5}, are represented in Figure \ref{fig:f7}, together with the carrying capacities $1+f_1(t)$ and $1+f_2(t)$, as well as the periodic coexistence state $(\tilde{u}, \tilde{v})$.\\
\begin{figure}[htp]
    \centering
    \includegraphics[width=14 cm]{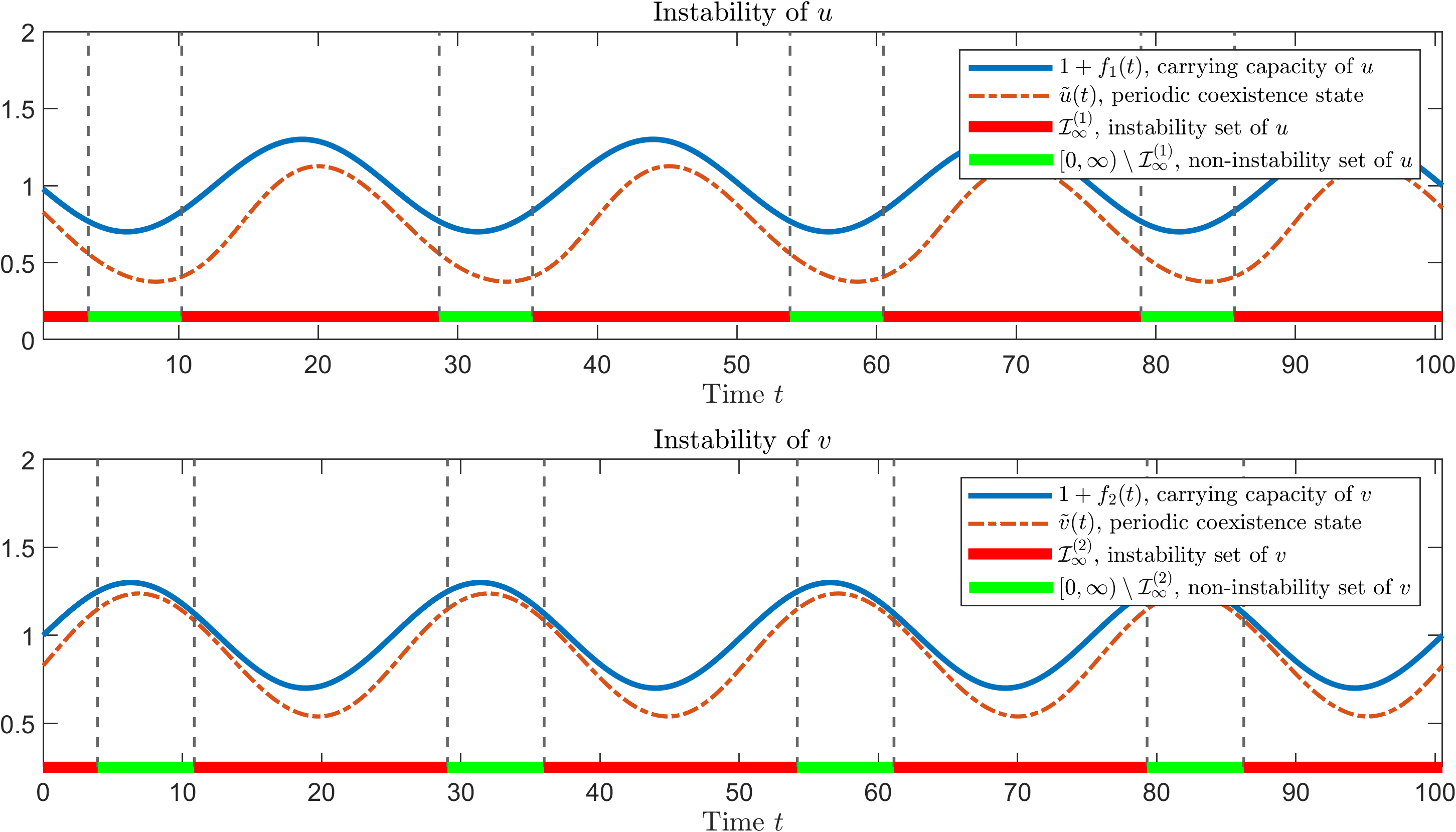}
    \caption{Location of the extended instability sets $\mathcal{I}_\infty^{(1)}$ and $\mathcal{I}_\infty^{(2)}$ and their complements for Example II, with parameters \eqref{5.4} and $\chi_1 = 5 > \chi_{\text{crit}}$. The carrying capacities $1+f_i(t)$ and the periodic coexistence state $(\tilde{u}, \tilde{v})$ are also included for reference.}
    \label{fig:f7}
\end{figure}\\
As it can be seen, and as indicated by \eqref{5.5}, the instability sets for $u$ and $v$ are very similar in terms of their temporal location. However, this contrasts with the out-of-phase behavior exhibited by $\tilde{u}$ and $\tilde{v}$, as a result of the non-synchronized carrying capacities $1+f_1$ and $1+f_2$. In particular, on the top panel, the complement set $[0,\infty)\setminus\mathcal{I}_\infty^{(1)}$, represented in green, is concentrated around the minima of the carrying capacity $1+f_1(t)$. Consequently, linear stability may occur only under relatively unfavorable environmental conditions for the first species, while instability arises when the carrying capacity is large enough. In contrast, the bottom panel reveals a markedly different situation, as the green regions, now corresponding to $[0,\infty)\setminus\mathcal{I}_\infty^{(2)}$, are located near the maxima of $1+f_2(t)$, where environmental conditions are most favorable for the second species.
\\\\
Hence, as previously anticipated, instability is not necessarily associated with intervals in which the carrying capacities attain sufficiently large values, promoting population growth. In fact, the dynamics of $v$ in this example reveal the complete opposite. Indeed, the location of these instability sets is not determined exclusively by ecological favorability, and instead follows from a more intricate mechanism rooted in the geometry of $\chi_1^{(k)}$ and $\widehat{\chi}_1^{(k)}$, as will be further discussed in Section \ref{in-set}. 
\\\\
To conclude this example, using initial data \eqref{5.3} ---the same values used in Example I--- we numerically solve system \eqref{0.8} with parameters \eqref{5.4} and $\chi_1 = 5 > \chi_{\text{crit}}$. This results in the plots shown in Figure \ref{fig:f8}, where, as in Figure \ref{fig:f4}, the comparative between $(u,v,w)$ and $(\tilde{u}, \tilde{v}, \tilde{w})$ is represented.
\begin{figure}[htp]
    \centering
    \includegraphics[width=17 cm]{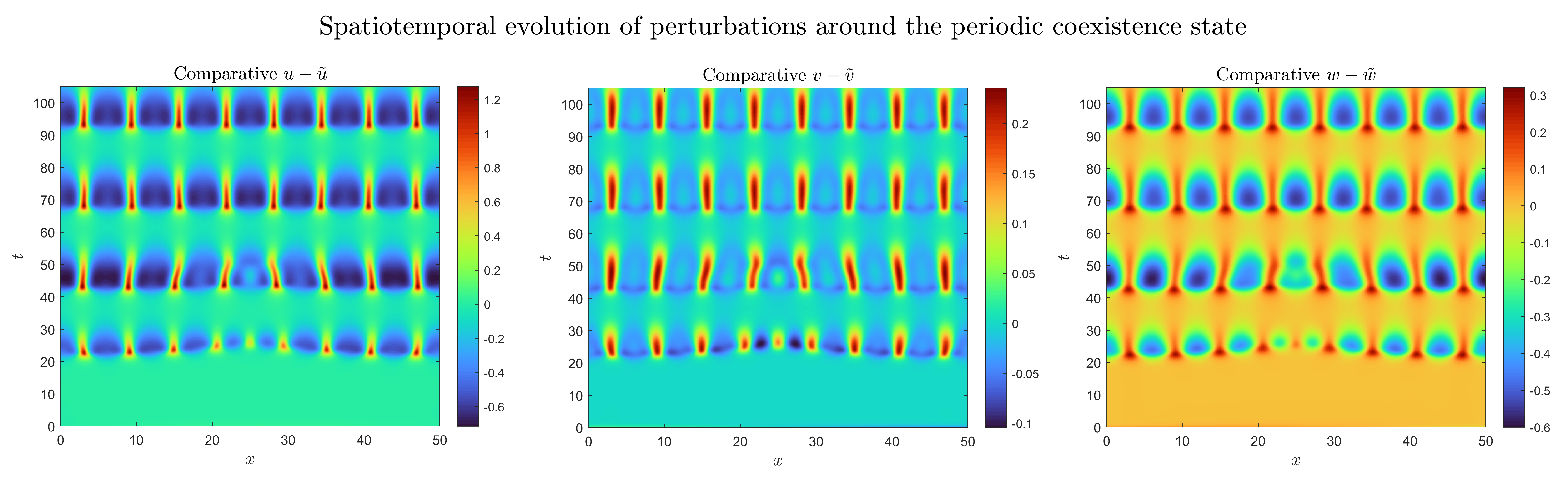}
    \caption{Comparison $u - \tilde{u}$, $v - \tilde{v}$, and $w - \tilde{w}$, where $(u,v,w)$ is the solution to system \eqref{0.8} with parameters \eqref{5.4}, $\chi_1 = 5 > \chi_{\text{crit}}$ and initial data \eqref{5.3}; $(\tilde{u}, \tilde{v})$ is the unique periodic coexistence state of system \eqref{0.9}; and $\tilde{w}(t) := \alpha \tilde{u}(t) + \beta \tilde{v}(t)$.}
    \label{fig:f8}
\end{figure}\\
As in Figure \ref{fig:f4}, we find that both species exhibit a cyclic aggregation--homogenization behavior, marked by the instability intervals $\mathcal{I}_\infty^{(1)}$ and $\mathcal{I}_\infty^{(2)}$. After an initial regime in which the solution remains close to the spatially homogeneous state, the first aggregates emerge during a relatively short time interval, approximately between $t = 20$ and $t = 30$. This is followed by recurrent stages of spatial homogenization and aggregation. In contrast with the spatial patterns obtained in Figure \ref{fig:f4}, the structures that arise in this case remain comparatively stationary after the first two environmental cycles, exhibiting only minor temporal variations before rapidly dissipating near the end of each instability interval.

\subsection{Competitive exclusion}

After the previous two examples, in which the ODE system \eqref{0.9} admitted a unique positive and globally attracting coexistence state as a result of \eqref{1.5}, we now investigate the competitive exclusion scenario. We select a parameter set for which \eqref{1.6} holds, and consequently the unique global attractor of system \eqref{0.9} is the semitrivial state $\big(\tilde{u}(t), 0\big)$, where $\tilde{u}$ is the unique positive and $T-$periodic solution to the nonautonomous logistic equation
$$
u' = \mu_1 u\big(1+f_1(t) -u\big), \quad t>0.
$$
\subsubsection{Example III: Semitrivial periodic state}
In order to compare the results with those obtained in the coexistence scenario, we select the same domain $\Omega = (0,50)$ and parameters from Example I, except for $a_2$, which we suitably modify so that \eqref{1.6} is satisfied. In this way, we consider
\begin{equation}\label{5.7}
    \begin{gathered}
a_1 = 0.25, \quad a_2 = 4, \quad \mu_1 = 1, \quad \mu_2 = 2.5,
\quad \alpha = 1, \quad \beta = 0.8,\\[1.5ex]
\chi_2 = 1.5, \quad 1+f_1(t) = 1-0.6\sin(0.25t),
\quad 1+f_2(t) = 1-0.5\sin(0.25t)\,.
\end{gathered}
\end{equation}
This time, the very high competition rate $a_2 = 4$ implies that
$$
1+f_2^M = 1.5 \leq 4 \cdot 0.4 = a_2(1+f_1^L),
$$
and as a result system \eqref{0.9} cannot sustain a population of the second species that does not eventually become extinct. Computing $\tilde{u}$, one can verify that the solvability assumption \eqref{0-hip2} holds for $\tilde{v} \equiv 0$.
\\\\
We recall that, as asserted by Lemma \ref{l2-2}, in the competitive exclusion case, $\widehat{\mathcal{Q}}_k(\chi_1,t)>0$ for all $k \in \mathbb{N}^+$, all $t>0$ and all $\chi_1 \in \mathbb{R}$. Consequently,
the instability analysis reduces to the study of $\chi_1^{(k)}(t)$, since there are no thresholds associated with $\widehat{\mathcal{Q}}_k$.
Computing the functions $\chi_1^{(k)}(t)$ for $k \in \{1,\dots, 30\}$, we represent the results in Figure \ref{fig:f9}, only showing the same subset of modes as in Figure \ref{fig:f2}. This time, the value of $\chi_{\text{crit}}$ is given by
$$
\chi_{\text{crit}} = \chi_c^{(18)} \approx 3.202,
$$
which is also represented in Figure \ref{fig:f9} in gray dashed line in the central and right-hand panels. As in Example I, we consider $\chi_1 = 6 > \chi_{\text{crit}}$, depicted in black dashed line.\\
\begin{figure}[htp]
    \centering
    \includegraphics[width=15 cm]{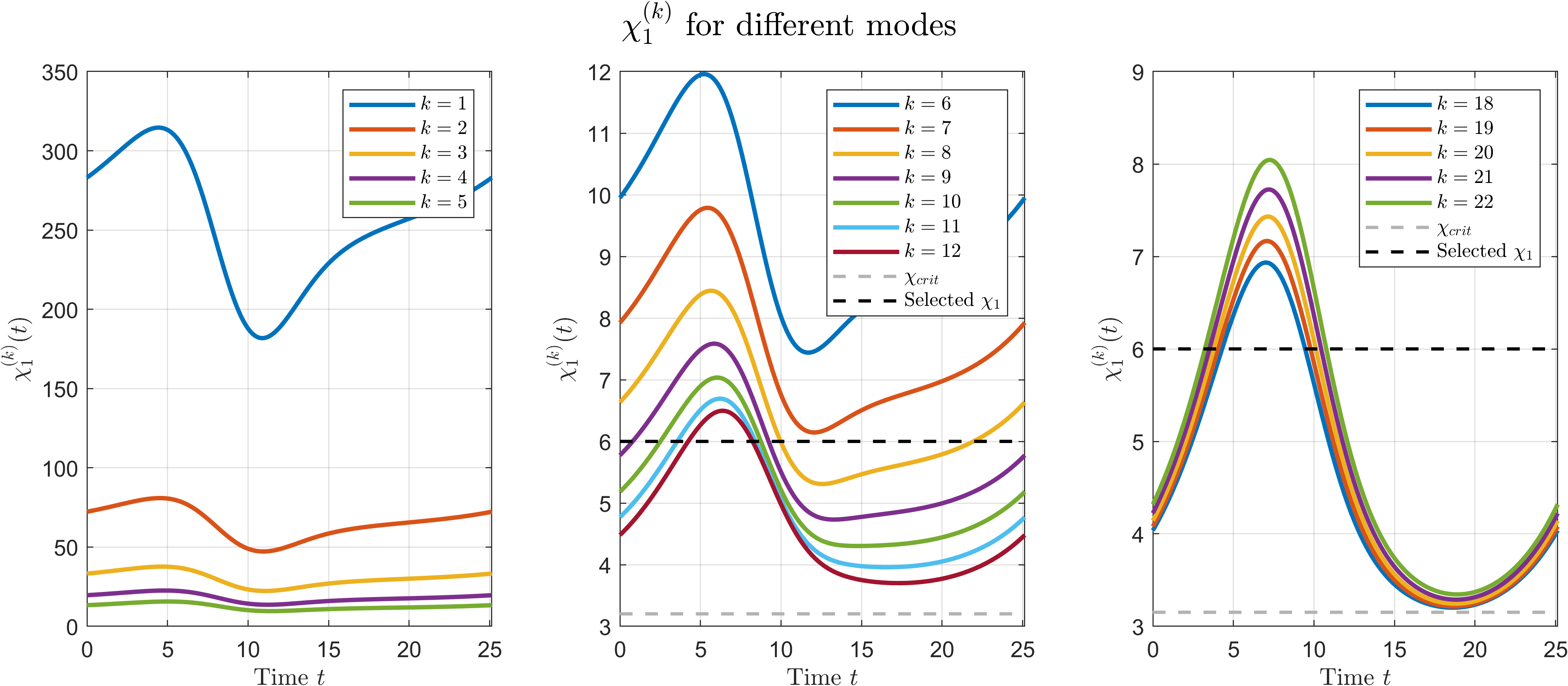}
    \caption{Values of $\chi_1^{(k)}(t)$ for different modes $k$, computed using the parameter set \eqref{5.7}. The value of $\chi_{\text{crit}} \approx 3.202$ and $\chi_1 = 6 > \chi_{\text{crit}}$ are also represented.}
    \label{fig:f9}
\end{figure}\\
Comparing the plots of $\chi_1^{(k)}(t)$ with those on the top panel of Figure \ref{fig:f2}, corresponding to the periodic coexistence case, we observe a very similar qualitative behavior. In particular, for each fixed mode $k$, the temporal profiles have essentially the same shape, with the main difference however lying in their magnitude. Throughout the range of modes considered, the values of $\chi_1^{(k)}(t)$ are consistently lower in the competitive exclusion case. For example, the maximum of $\chi_1^{(1)}(t)$ in Figure \ref{fig:f9} is below $350$, compared with approximately $375$ in the coexistence case from Figure \ref{fig:f2}, while for $k=6$, in Figure \ref{fig:f9} we find $\chi_1^{(6)}(t)$ uniformly below $12$, whereas it attains values close to $14$ in Figure \ref{fig:f2}. This difference is due to the combination of contribution of $\tilde{v}$ in the coexistence case, which vanishes in the current competitive exclusion scenario, and the effect of the higher value of $a_2$.
\\\\
This further influences the critical threshold for instability $\chi_{\text{crit}}$, given in this case by $3.202$, which is slightly less than $3.287$, obtained for the coexistence case. Consequently, in the present example, where the two parameter sets differ only in the value of $a_2$, this indicates that increasing the competitive effect of the first species on the second lowers the chemotactic sensitivity required for transient instability.
\\\\
With the chosen value of $\chi_1 = 6 > \chi_{\text{crit}}$, the instability set is given by
\begin{equation}\label{5.8}
    \mathcal{I} = \mathcal{I}_1 = [0, 4.925) \cup (8.083, 25.132].
\end{equation}
As in the previous examples, we represent the extended instability set $\mathcal{I}_\infty^{(1)}$, together with the carrying capacity $1+f_1(t)$ and the periodic component $\tilde{u}(t)$ of the unique global attractor to system \eqref{0.9} in Figure \ref{fig:f10}. The corresponding quantities associated with the second species are omitted, since in the case of competitive exclusion, the sufficient instability criterion does not detect instability of the trivial component.\\
\begin{figure}[htp]
    \centering
    \includegraphics[width=14 cm]{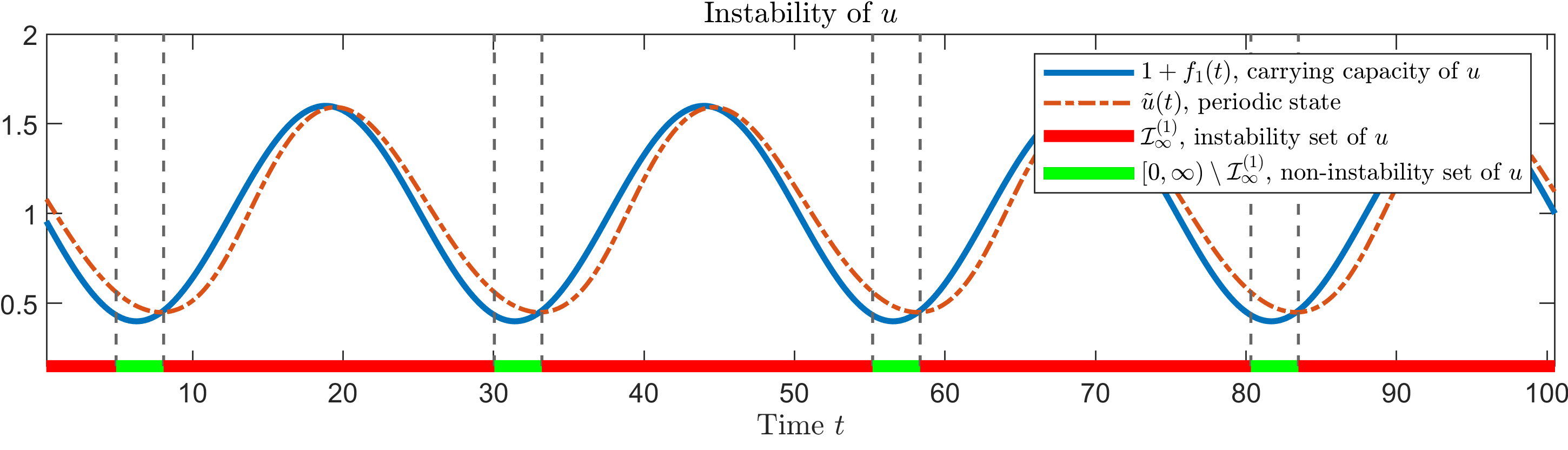}
    \caption{Location of the instability set $\mathcal{I}_\infty^{(1)}$ and its complement, together with $\tilde{u}$, the nonzero periodic component of the semitrivial state, and $1+f_1(t)$, the carrying capacity of the first species, computed with parameters \eqref{5.7}.}
    \label{fig:f10}
\end{figure}\\
This time, as Figure \ref{fig:f3} showed for Example I, the only possible regimes of stability are located in a neighborhood of the minima of the carrying capacity of $1+f_1$, while for sufficiently large carrying capacities, instability holds.
\\\\
Lastly, using this value of $\chi_1 = 6 > \chi_{\text{crit}}$ and parameters \eqref{5.7}, we solve the complete PDE system \eqref{0.8} with initial data
\begin{equation}\label{5.9}
    u(x,0) = \tilde{u}(0) + \varepsilon \sin (\pi x/50), \quad v(x,0) = \tilde{v}(0) + \varepsilon \big| \hspace{-0.05 cm}\cos (\pi x/50)\big| = \varepsilon \big|\hspace{-0.05 cm}\cos (\pi x/50)\big|,
\end{equation}
for $\varepsilon = 0.05$. We remark that these are the same initial data considered in \eqref{5.4}, except for the fact that the perturbation in $ v(x,0)$ is considered in absolute value to preserve the nonnegativity of the solution. The evolution plots are represented in Figure \ref{fig:f11}, showing the comparative between $(u,v,w)$ and $(\tilde{u}, \tilde{v}, \tilde{w})$.\\
\begin{figure}[htp]
    \centering
    \includegraphics[width=17 cm]{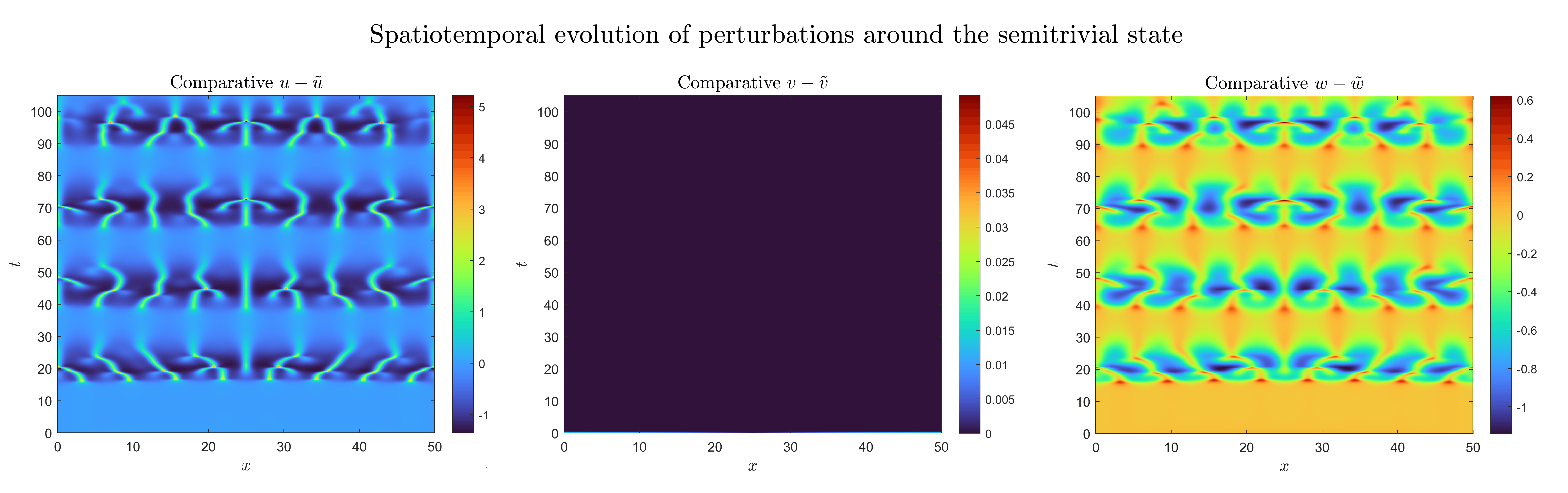}
    \caption{Comparison $u - \tilde{u}$, $v - \tilde{v}$, and $w - \tilde{w}$, where $(u,v,w)$ is the solution to system \eqref{0.8} with parameters \eqref{5.7}, $\chi_1 = 6 > \chi_{\text{crit}}$ and initial data \eqref{5.9}; $(\tilde{u}, \tilde{v}) = (\tilde{u},0)$ is the unique global attractor of system \eqref{0.9}; and $\tilde{w}(t) := \alpha \tilde{u}(t) + \beta \tilde{v}(t) = \alpha \tilde{u}(t)$.}
    \label{fig:f11}
\end{figure}\\
The most substantial difference with respect to the previous examples is observed in the $v$ component of the solution. The initial perturbation rapidly decays, as a result of the strong competitive effect exerted by the first species, and consequently the center panel shows essentially no discrepancies between $v$ and $\tilde{v}$. We recall that, in the competitive exclusion case, Lemma \ref{l2-2} ensures the positivity of $\widehat{\mathcal{Q}}_k$ for all $k \in \mathbb{N}$, preventing the application of the sufficient instability criterion from Lemma \ref{l1-aux} to the modal amplitudes of the perturbations of $v$. Although this does not rule out the possibility of instability, for the present parameter set, perturbations of the excluded species decay rapidly, so that competitive exclusion is maintained despite the transient instability exhibited by the first species.
\\\\
With respect to $u$ and $w$, transient aggregation patterns once again develop, illustrating the cyclic behavior, alternating heterogeneous aggregations phases ---occurring during the extended instability set--- with periods of spatial homogeneity, where perturbations decay. Despite their overall symmetry, the resulting patterns are less regular than those observed in Figures \ref{fig:f4} and \ref{fig:f8}, exhibiting thinner aggregates that attain considerably higher population densities. In particular, the density of the first species exceeds 5, substantially above the values obtained in Example I, despite the same parameters (and particularly the same chemotactic sensitivity) being considered for both $u$ and $w$. This difference can be attributed to the competitive exclusion of the second species, which removes the interspecific competition acting on the first one and allows it to concentrate more strongly within the aggregation regions.

\section{Further remarks}\label{s6}
Lastly, we collect here some final observations concerning the results from the previous examples.

\subsection{On the location of the instability sets}\label{in-set}

Despite what Figure \ref{fig:f3} in Example I might suggest, and as illustrated by Figure \ref{fig:f7} in Example II, the location of the instability sets is not directly tied to the times at which the carrying capacities attain their maxima. Instead, for a fixed value of $\chi_1 > \chi_{\text{crit}}$ the instability sets $\mathcal{I}_1$ and $\mathcal{I}_2$ are fundamentally determined by the geometry of the threshold functions $\chi_1^{(k)}$ and $\widehat{\chi}_1^{(k)}$, whose shape ultimately depends on the coefficients $a_k$, $b_k$ and $\widehat{b}_k$ introduced in Lemma \ref{l2}.
\\\\
In particular, one important observation can be made with respect to $a_k$ and the role of hypothesis \eqref{0-hip2}. From \eqref{a-k-eq}, we have that
$$
 a_k(t) = \tilde{u}(t) \frac{ \lambda_k}{1+\lambda_k} \Big[- \alpha\lambda_k + \mu_2  \big(\alpha (1+f_2(t)) - G(t)\big) \Big ],
$$
where 
$$
G(t) :=  \alpha a_2  \tilde{u}(t)  + \left(2\alpha -a_2\beta \right) \tilde{v}(t) 
$$
As previously discussed, assumption \eqref{0-hip2} implies that $\alpha(1+f_2(t)) - G(t) <0$ for all $t \in [0,T]$, ensuring that $a_k(t)$ remains uniformly negative over $[0,T]$. Furthermore, for each fixed mode $k \in \mathbb{N}^+$, the temporal profile of $a_k$ ---and consequently those of $\chi_1^{(k)}$ and $\widehat{\chi}_1^{(k)}$--- is largely dependent on that of $G(t)$. Coefficients $\alpha a_2$ and $2\alpha -a_2 \beta$ can therefore be understood as two weights, determining the relative contribution of $\tilde{u}$ and $\tilde{v}$ to $G(t)$
\\\\
In Example I, as the carrying capacities were in phase, the periodic coexistence states $\tilde{u}$ and $\tilde{v}$ inherit an essentially similar behavior, as Figure \ref{fig:f3} shows. Hence, the weighted sum $G(t)$ preserves this common temporal structure, and its extrema are therefore aligned with those of $1+f_2$, as depicted in Figure \ref{fig:f1}. Similarly, since the time dependence of the coefficients $b_k$ and $\widehat{b}_k$ is expressed only through $\tilde{u}$, $\tilde{v}$ and the carrying capacities $f_i$, the threshold functions $\chi_1^{(k)}$ and $\widehat{\chi}_1^{(k)}$, which by \eqref{chis-def} are given by
$$
\chi_1^{(k)}(t) = - \frac{b_k(t)}{a_k(t)}, \quad \widehat{\chi}_1^{(k)} = - \frac{\widehat{b}_k(t)}{a_k(t)},
$$
inherit essentially the same temporal phase. Consequently, the instability sets \eqref{5-i2} were concentrated around the maxima of the periodic states.
\\\\
In Example II, the out-of-phase carrying capacities lead to a markedly different behavior. As shown in Figure \ref{fig:f7}, the instability set of $u$ remains concentrated around the maxima of $\tilde{u}$, whereas that of $v$ is now located near the minima of $\tilde{v}$ . 
This time, for the parameter set \eqref{5.4}, the weighted sum takes the form
$$
G(t) = 0.15 \tilde{u}(t) + 1.85 \tilde{v}(t),
$$
and thus its behavior is dominated by the contribution of $\tilde{v}$. Since the phase shift in the carrying capacity $1+f_2$ reverses the oscillatory behavior of $\tilde{v}$ with respect to Example I, the same phase reversal is inherited by $G(t)$, and consequently by the dominant contribution to the coefficient $a_k$. As a result, the threshold functions attain their smallest values near the minima of $\tilde{v}$, leading to instability intervals concentrated around these times.
\\\\
In this way, $G(t)$ provides a mechanism for controlling the temporal location of the instability intervals. A further important observation is that the geometry of $G$ is not determined only by the weights, but also by the magnitude of the coexistence states $\tilde{u}$ and $\tilde{v}$. In Example II, as both carrying capacities attained the same maximum and minimum values, and competition between both species was mild, the magnitude of both states $\tilde{u}$ and $\tilde{v}$ was similar, as Figure \ref{fig:f7} shows. However, if one species experiences a much stronger competitive pressure or significantly less favorable environmental conditions, its population may remain comparatively small, reducing its influence on $G(t)$, even when its associated weight is relatively large.
\\\\
As a result, by suitably balancing the weights in $G(t)$ and the amplitudes of $\tilde{u}$ and $\tilde{v}$, while simultaneously ensuring that hypothesis \eqref{0-hip2} holds, one can obtain a wider variety of instability distributions. In particular, the extrema of $G(t)$ need not coincide with those of either carrying capacity. To illustrate this behavior, we consider the following parameter values
\begin{equation}\label{6.1}
\begin{gathered}
    a_1 = 0, \quad a_2 = 0.3, \quad \mu_1 = 3, \quad \mu_2 = 0.25, \quad \alpha = 4, \quad \beta = 2.5, \\[1.5 ex]
    \chi_2 = 1.5, \quad 1+f_1(t) = 1.4 - 0.4\sin(0.25t), \quad 1+f_2(t) = 0.7+0.1\sin(0.25t),
\end{gathered}
\end{equation}
again with out-of-phase carrying capacities. In this case, the combination of the absence of interspecific competition exerted by the second species $(a_1 = 0)$ and the fact that $1+f_1(t)$ lies substantially above $1+f_2(t)$ for all times, results in $\tilde{u}$ attaining significantly higher values than $\tilde{v}$, as shown on the top panel of Figure \ref{fig:f12}. Consequently, although here $G(t) = 1.2 \tilde{u}(t)+ 7.25 \tilde{v}(t)$, the larger magnitude of $\tilde{u}$ together with the comparatively small amplitude of the oscillations of $\tilde{v}$ cause the phase of $G(t)$ to lie closer to that of $1+f_1(t)$ than of $1+f_2(t)$. This is represented in the bottom panel of Figure \ref{fig:f12}, where $G(t)$ is plotted in blue, together with $\alpha(1+f_2(t))$ in red. As it can be seen, this shift of phase in $G(t)$ makes the inequality in hypothesis \eqref{0-hip2} tighter. \\
\begin{figure}[htp]
    \centering
    \includegraphics[width=14cm]{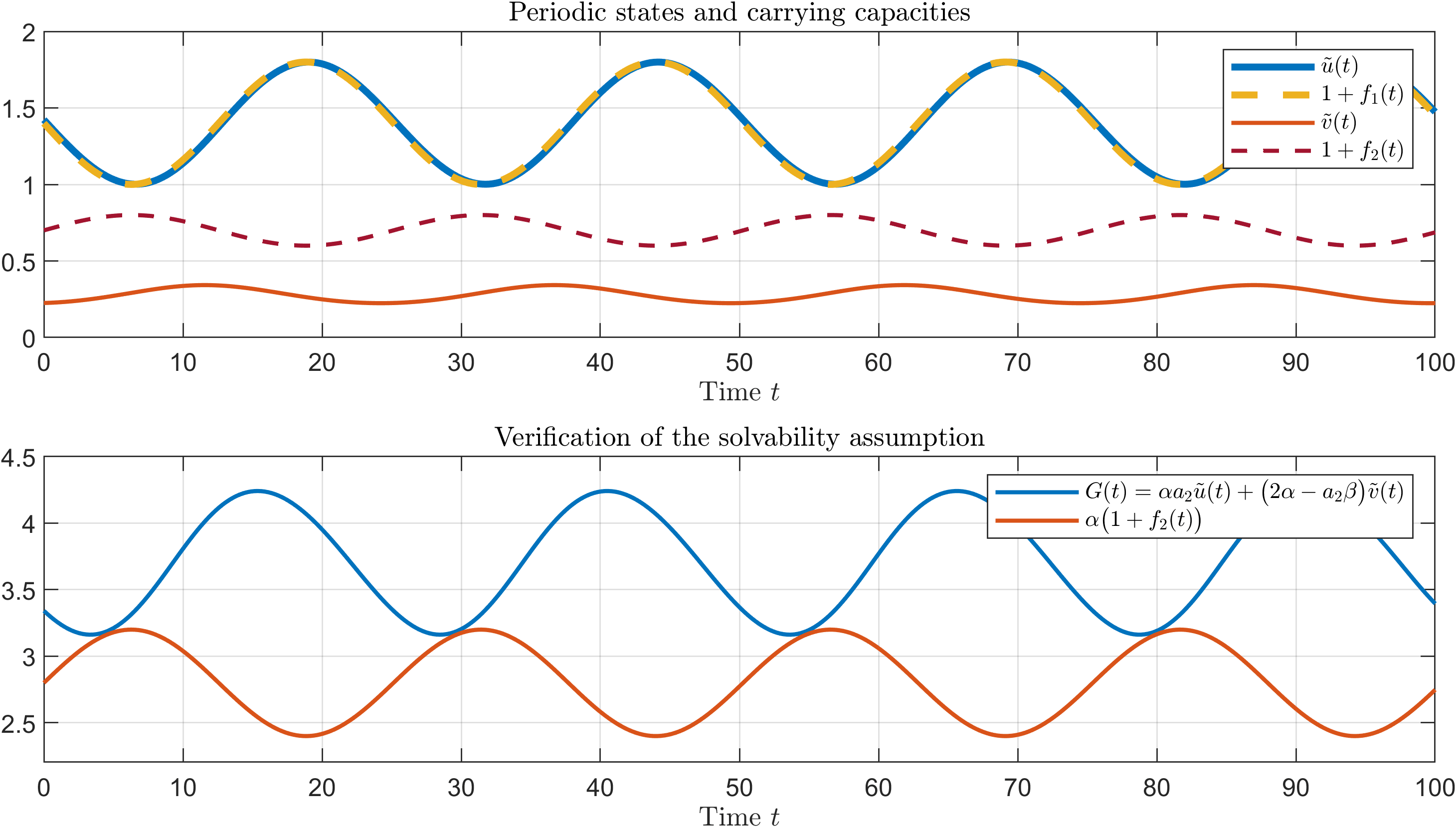}
    \caption{Top panel: periodic coexistence states $(\tilde{u},\tilde{v})$ and their carrying capacities $1+f_1(t)$ and $1+f_2(t)$ for parameter set \eqref{6.1}. Bottom panel: numerical verification of assumption \eqref{0-hip2} showing $G(t)$ and $\alpha (1+f_2(t))$.}
    \label{fig:f12}
\end{figure}
\\
Parameter values \eqref{6.1} yield $\chi_{\text{crit}} =  0.588$. Selecting for instance $\chi_1 =1.05 > \chi_{\text{crit}}$, produces the instability intervals depicted in Figure \ref{fig:f13}. As it can be seen, the extended instability intervals $\mathcal{I}_\infty^{(1)}$ and $\mathcal{I}_\infty^{(2)}$ now contain both, the maxima and minima of the carrying capacities. Moreover, the complement sets, where linear stability may occur, correspond to values of the carrying capacities for which instability is also observed at other times within the same environmental cycle. Hence, identical values of the carrying capacities may be associated with either instability or possible stability, depending on the stage of the environmental cycle at which they are attained.
\\
\begin{figure}
    \centering
    \includegraphics[width=14 cm]{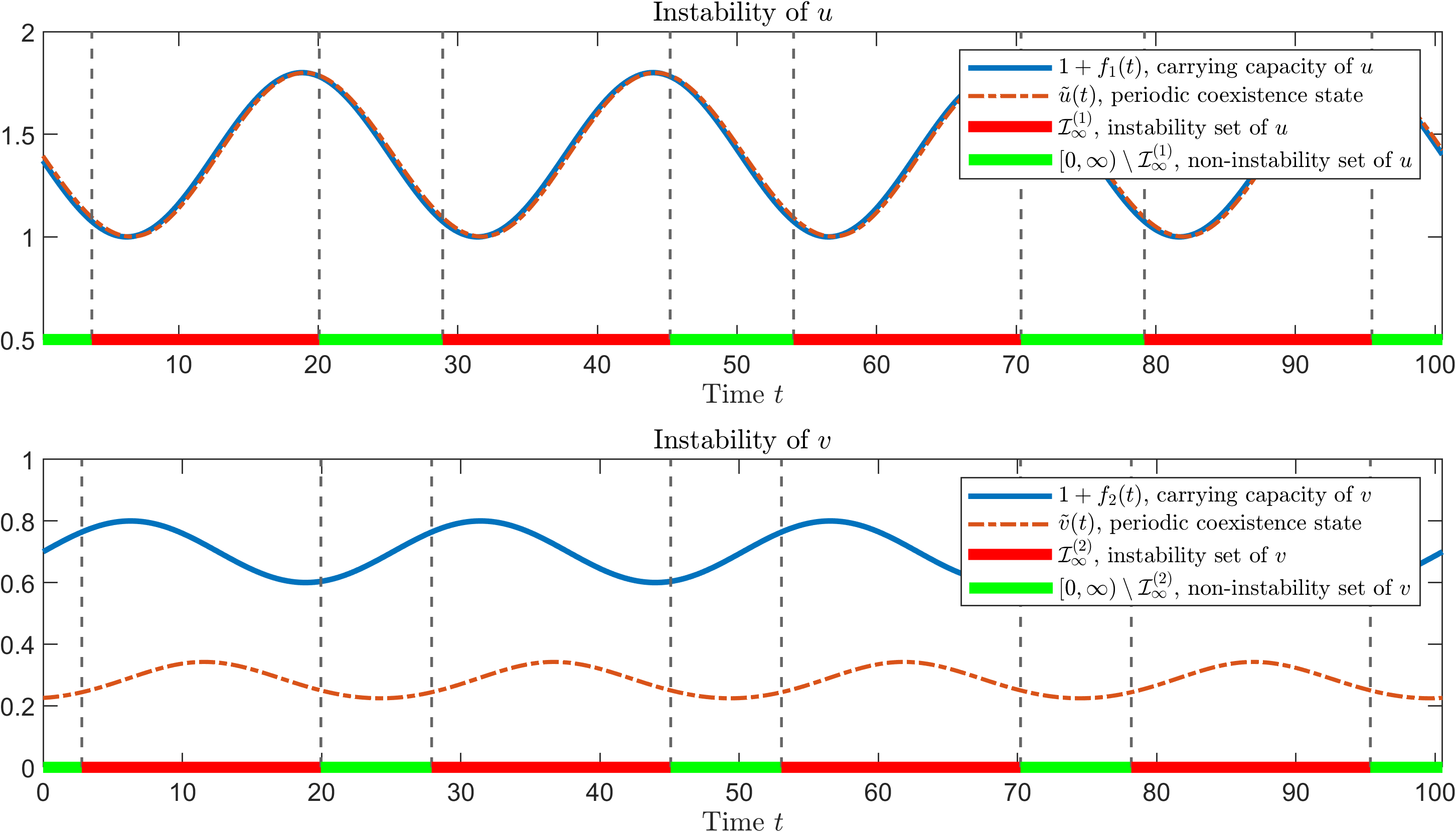}
    \caption{Location of the instability sets $\mathcal{I}_\infty^{(1)}$ and $\mathcal{I}_\infty^{(1)}$ and their complements for parameter values \eqref{6.1}.}
    \label{fig:f13}
\end{figure}
\\
This shows that instability is not determined by the instantaneous value of the carrying capacities alone, but also by the temporal interaction between the environmental forcing and the periodic coexistence state, mediated through the weighted function $G(t)$.

\subsection{Negative instability thresholds}\label{neg-chi}

Lastly, we devote a few lines to discuss the possibility of negative instability thresholds. Despite all the previous examples yielding positive values of $\chi_{\text{crit}}$, such positivity is not guaranteed for arbitrary parameter values. 
\\\\
For the threshold functions $\chi_1^{(k)}$ and $\widehat{\chi}_1^{(k)}$ defined in \eqref{chis-def}, hypothesis \eqref{0-hip2} only ensures that $a_k(t) < 0$ for all $t \in [0,T]$. However, the signs of  $\chi_1^{(k)}$ and $\widehat{\chi}_1^{(k)}$ ---and consequently that of $\chi_{\text{crit}}$--- depend naturally on the coefficients $b_k$ and $\widehat{b}_k$. Without additional assumptions, these coefficients need not remain positive throughout $[0,T]$.
\\\\
In particular, if there exists $k \in \mathbb{N}^+$ such that $b_k(t)$ or $\widehat{b}_k(t)$ become negative for some $t \in [0,T]$, then $\chi_{\text{crit}} <0$. Consequently, heterogeneous aggregations may arise when the first species is non-chemotactic, or even when it undergoes repulsive taxis, which is generally a further stabilizing mechanism. In such situations, the instability can be understood as being driven by the attractive chemotactic response of the second species, whose influence is sufficient to destabilize the homogeneous state despite the absence of attracting taxis in the first species.
\\\\
As an example, we consider the following parameters
\begin{equation}\label{6.2}
    \begin{gathered}
        a_1 = 0.4, \quad a_2 = 0.4, \quad \mu_1 = 3, \quad \mu_2 = 0.7, \quad \alpha = 5, \quad \beta = 7, \\[1.5 ex]
        \chi_2 = 1.5, \quad 1+f_1(t) = 3-0.5\sin(0.25t), \quad 1+f_2(t) = 1-0.5\cos(0.25t).
    \end{gathered}
\end{equation}
The left panel on Figure \ref{fig:f14} shows the plots of $\chi_1^{(k)}(t)$ and $\widehat{\chi}_1^{(k)}(t)$, for $k \in \{5, \dots, 10\}$. The values of $\chi_1^{(k)}$ are plotted in solid lines, while  $\widehat{\chi}_1^{(k)}$ are represented in dashed lines. As it can be seen, for each $k$, the values on both graphs are very similar, in particular attaining negative values. This results in the negative threshold
$$
\chi_{\text{crit}} = \widehat{\chi}_1^{(5)}(5.05)= -9.72 < 0.
$$
To illustrate the patterns generated for $\chi_1 > \chi_{\text{crit}}$ the right panel of Figure \ref{fig:f14} depicts two representative solution profiles, the first computed for $\chi_1 = -6 <0$, shown on the top row, and the second one for $\chi_1 = 0$, on the bottom row.
\begin{figure}[htp]
    \centering
    \includegraphics[width=16 cm]{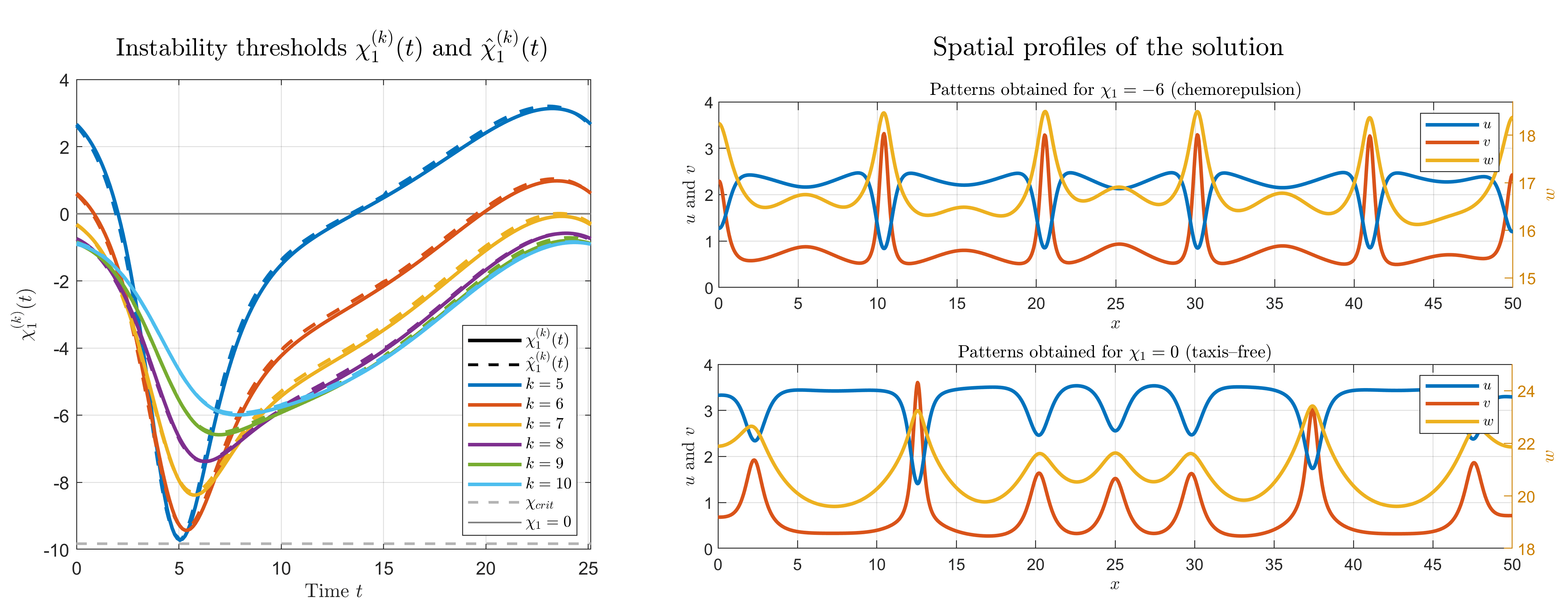}
    \caption{Results obtained for parameter set \eqref{6.2}. Left panel: plots of $\chi_1^{(k)}(t)$ (solid) and $\widehat{\chi}_1^{(k)}(t)$ (dashed lines). Right panel: on the top row, spatial profile of the patterns generated by $\chi_1 = -6 <0$, on the bottom row, patterns generated by $\chi_1 = 0$.}
    \label{fig:f14}
\end{figure}\\
In both cases, we obtain highly structured and nearly symmetric patterns. In the repulsive taxis case, the minima of $u$ coincide with the maxima of $w$, reflecting the chemorepulsive response of the first species. In contrast, these regions of high chemoattractant concentration promote the aggregation of the second species $v$, due to the attraction caused by the positive sensitivity $\chi_2 >0$. 
\\\\
A similar behavior is observed in the taxis--free scenario, where only the second species is chemotactically attracted to
$w$, while the first species evolves solely through diffusion, logistic growth, and interspecific competition. Although $u$ no longer exhibits chemorepulsion, its minima are again located at the maxima of $w$. In this case, however, the mechanism is different. In particular the high concentrations of $w$ attract the second species, leading to aggregations of $v$, which in turn suppress the first species through interspecific competition.
\\\\
We stress, however, that the instability mechanism remains fundamentally chemotactically driven. Indeed, if both sensitivities vanish, that is, $\chi_1 = \chi_2 = 0$, the result in \cite{AL89} rules out the possibility of instability. Consequently, cases with $\chi_{\text{crit}} \leq 0$ should not be interpreted as pattern formation in the absence of chemotaxis. Rather, they reflect that the attractive chemotactic response of the second species alone is already sufficient to destabilize the spatially homogeneous state, even when the first species is non-chemotactic or repelled by the chemoattractant.

\section{Discussion and conclusions}\label{s7}

Throughout this work, we investigated the rise of recurrent aggregations in the two--species chemotaxis--competition system \eqref{0.8}, which features periodic environmental conditions. Our aim was to derive a threshold for $\chi_1$, the chemotactic sensitivity of the first species, that leads to transient linear instability of the spatially homogeneous state $\big(\tilde{u}(t), \tilde{v}(t), \tilde{w}(t)\big)$, obtained through the ODE system \eqref{0.9}. This results in an opposite behavior to that studied in \cite{NV21}, where it was shown that sufficiently large logistic growth rates $\mu_1$ and $\mu_2$ guarantee convergence of all solutions of \eqref{0.8} to $(\tilde{u}, \tilde{v}, \tilde{w})$. In contrast, the present work identifies parameter regimes in which transient linear instability and recurrent aggregation may occur.
\\\\
In this way, our main analytical result is Theorem \ref{t1}, which provides a way to compute the critical threshold $\chi_{\text{crit}}$. In particular, for any $\chi_1 > \chi_{\text{crit}}$, there exists a set of times $\mathcal{I} \subset [0,T]$, over which the spatially homogeneous state $(\tilde{u}, \tilde{v}, \tilde{w})$ becomes linearly unstable.
\\\\
Particularly interesting dynamics arise whenever $[0,T] \setminus \mathcal{I}$ is not empty, as due to the environmental periodicity, the instability becomes recurrent over the subsequent environmental cycles. In particular, the examples in Section \ref{s4} showed the alternation of stable and unstable intervals, over which spatial perturbations grow and decay, resulting in transient aggregations. This interplay between homogenization and destabilization is clearly visible in the plots shown in Figures \ref{fig:f4} and \ref{fig:f8}, in the case of periodic coexistence, and Figure \ref{fig:f11} in a competitive exclusion scenario. 
\\\\
Some final remarks were made in Section \ref{s6} concerning the location of the instability sets and the possibility of negative thresholds for $\chi_1$. In particular, the example constructed through parameters \eqref{6.1} showed how instability is not directly tied to instantaneous values of the carrying capacities, and in fact the same value of a carrying capacity may correspond either to stability or to instability, depending on the stage of the environmental cycle. With respect to negative values of $\chi_{\text{crit}}$, parameter set \eqref{6.2} produced the results shown in Figure \ref{fig:f14}, generating patterns in taxis--free and even chemorepulsive scenarios.
\\\\
We emphasize that the instability mechanism considered throughout this work is fundamentally chemotactically driven, as remarked at the end of Section \ref{s5}. While the ecological quantities $1+f_i$, $\tilde{u}$ and $\tilde{v}$ determine the conditions under which instability arises, the underlying destabilizing mechanism remains chemotactic.
\\\\
Finally, we note that the instability criterion developed in this work is based on the sufficient condition from Lemma \ref{l1-aux}. Consequently, although it provides a practical framework for detecting transient instability, it cannot be used to conclude linear stability. Moreover, the present analysis is entirely linear and therefore concerns only the onset of pattern formation. While the examples showed that it successfully predicts the emergence of transient aggregations and their temporal location, it does not provide information on the nonlinear evolution of the patterns, their shape or the interactions between aggregates once instability has developed.

\section*{Acknowledgments}

This work was supported by Project PID2022-141114NB-I00 from the Spanish Ministry of Science and Innovation and by Grant FPU23/03170 from the Spanish Ministry of Science, Innovation and Universities (F.H.-H.)

\end{document}